\documentclass[12pt, a4paper]{amsart}
\usepackage{amssymb}
\usepackage{tikz} 
\usepackage{mathabx}
\usepackage{graphicx}
\usepackage{mathrsfs}
\usepackage{hyperref}
\usepackage{import}
\usepackage{tcolorbox}

\theoremstyle{plain}

\newtheorem{theorem}{Theorem}[section]
\newtheorem{lemma}[theorem]{Lemma}
\newtheorem{proposition}[theorem]{Proposition}

\theoremstyle{definition}

\newtheorem{define}{Definition}[section]
\newtheorem{problem}{Problem}

\newtheorem{example}{Example}[section]

\theoremstyle{remark}
\newtheorem{remark}{Remark}[section]%numbering within section

\begin{document}

\date\today

\title[\large P\MakeLowercase{roperties of squeezing functions}]{ \large P\MakeLowercase{roperties of squeezing functions on $h$-extendible domains}}
\author{Ninh Van Thu}

\address{Ninh Van Thu}
\address{Faculty of Mathematics and Informatics, Hanoi University of Science and Technology, No. 1 Dai Co Viet, Bach Mai, Hanoi, Vietnam}
\email{thu.ninhvan@hust.edu.vn}

%\pacs[MSC Classification]{32M05, 32H02, 32F18}

\begin{abstract}
The purpose of this article is twofold. We first prove that the localized squeezing function approaches 1 near strongly pseudoconvex boundary points of bounded domains in $\mathbb{C}^{n+1}$. Then, we show that the localized squeezing function approaches 1 along certain sequences converging to pseudoconvex boundary points of finite type in the sense of \cite{D'A82}, including uniformly $\Lambda$-tangential and spherically $\frac{1}{2m}$-tangential convergence patterns.
\end{abstract}

\maketitle

\section{Introduction}
Let $\Omega$ be a  bounded domain in $\mathbb{C}^{n+1}$ and we denote by $\mathrm{Aut}(\Omega)$ the set of all automorphisms of $\Omega$. For a point $p \in \Omega$ and a holomorphic embedding $f\colon \Omega \to \mathbb{B}^{n+1}= B(0,1)$ with $f(p)=(0',0)$, one sets
$$
\sigma_{\Omega}(p, f):=\sup\left \{r>0\colon B(0,r)\subset f(\Omega)\right\}.
$$
Here and in what follows, $B(z,r)$ denotes the ball centered at $z\in \mathbb{C}^{n+1}$ with radius $r>0$. Then the \textit{squeezing function} $\sigma_{\Omega}\colon \Omega\to\mathbb{R}$ is defined as
$$
\sigma_{\Omega}(p):=\sup_{f} \left\{\sigma_{\Omega}(p, f)\right\}.
$$
(See Definition $1.1$ in \cite{DGZ16}.) Notice that the squeezing function is invariant under biholomorphisms and $0 < \sigma_{\Omega}(z)\leq 1$ for any $z \in \Omega$. Furthermore, $\Omega$ is biholomorphically equivalent to the unit ball $\mathbb{B}^{n+1}$ if $\sigma_{\Omega}(z)=1$ for some $z \in \Omega$.
 
We would like to emphasize that the squeezing function $\sigma_{\Omega}$ has a fundamentally global character: its value at a point $p \in \Omega$ depends on holomorphic embeddings of the entire domain $\Omega$ into the ball, not merely on the local geometry near $p$. Therefore, for a boundary point $\xi_0 \in \partial\Omega$, it is more natural to study the boundary behavior of \emph{the localized squeezing function}, that is, $\sigma_{\Omega \cap U}$ for a suitable neighborhood $U$ of $\xi_0$, rather than that of $\sigma_{\Omega}$ itself.  For the localization of the squeezing function, we refer the reader to \cite{FN22,NV20,RY22}.

In this paper, we investigate the phenomenon of the localized squeezing function approaching $1$ along sequences converging to pseudoconvex boundary points by the scaling technique, introduced by S. Pinchuk (cf. \cite{Pi91}). For a smoothly bounded planar domain $D$ and $p \in \partial D$, one always has $\lim\limits_{z \to p} \sigma_D(z) = 1$ (see \cite{DGZ12}). For pseudoconvex domains of higher dimension, this result holds when $D$ is a bounded strongly pseudoconvex domain in $\mathbb{C}^{n+1}$ (cf. \cite{DGZ16}), or when $p$ is globally strongly convex in the sense of \cite{DGZ16}, or when $p$ is a spherically extreme boundary point in the sense of \cite{KZ16}.
 
The first aim of this paper is to prove the following theorem, which shows that the localized squeezing function approaches $1$ along any sequence converging to a strongly pseudoconvex boundary point.
\begin{theorem}\label{maintheorem1}
Let $\Omega$ be a bounded domain in $\mathbb{C}^{n+1}$ and $\xi_0\in \partial \Omega$. If $\partial \Omega$ is $\mathcal{C}^2$ strongly pseudoconvex at $\xi_0$, then we have $\lim\limits_{z\to \xi_0}\sigma_{\Omega\cap U_0}(z)=1$, where $U_0$ is a sufficiently small  neighborhood of $\xi_0$.
\end{theorem}
\begin{remark} 
We note that Theorem \ref{maintheorem1} is a local version of \cite[Theorem $1.3$]{DGZ16}. Moreover, if $\xi_0$ is a spherically extreme boundary point, then $\xi_0$ is strongly pseudoconvex (cf. \cite{NTN25}), and hence Theorem $3.1$ in \cite{KZ16} follows from Theorem \ref{maintheorem1}.  
\end{remark}

Next, we investigate the asymptotic behavior of the squeezing function near weakly pseudoconvex points. Namely, related to Problem $4.1$ in \cite{FW18} and Question $8.1$ in \cite{Bh23}, let us consider the following problem.
\begin{problem}\label{prob1}
If $\Omega$ is a bounded pseudoconvex domain with $\mathcal{C}^\infty$-smooth boundary, and if $\lim\limits_{j\to\infty}\sigma_{\Omega}(\eta_j)=1$ for some sequence $\{\eta_j\}\subset \Omega$ converging to $p\in \partial \Omega$, then is the boundary of $\Omega$ strongly pseudoconvex at $p$?
\end{problem}

It is well-known that the answer to this problem is affirmative, provided that $\partial\Omega$ is $h$-extendible with Catlin's finite multitype $(2m_1,\ldots, 2m_n, 1)$ at $p$ and $\{\eta_j\}$ converges $\Lambda$-non-tangentially to $p$ (cf. Definition $3.4$ in \cite{NN20}), where the multi-weight $\Lambda=\left(\frac{1}{2m_1},\ldots, \frac{1}{2m_{n}}\right)$ (cf. \cite{JK18, Ni18, MV19, NN20,NNC21}). 

We note that A. Zimmer \cite{Zi18} first proved that if $\lim_{z\to \partial D}\sigma_{D}(z)=1$ for a $\mathcal{C}^\infty$-smooth bounded convex domain $D$, then the domain is necessarily strictly pseudoconvex. Subsequently, in contrast to Problem \ref{prob1}, a bounded convex $\mathcal{C}^2$-smooth domain $\Omega\subset\mathbb{C}^{n+1}$ that is not strongly pseudoconvex was constructed in \cite{FW18}, satisfying $\displaystyle \lim_{z\to \partial \Omega}\sigma_{\Omega}(z)=1$. Moreover, the authors \cite[Example $1.2$]{NNN24} showed that $\sigma_{E_{1,2}}(\eta_j)\to 1$ as $j\to\infty$ for the sequence 
$$E_{1,2}\ni \eta_j=\left(\sqrt[4]{\frac{2}{j}-\frac{2}{j^2}}, 1-\frac{1}{j}\right)$$
converging tangentially to the non-strictly pseudoconvex boundary point $(0,1)$, where $E_{1,2}:=\{(z_1,z_2)\in \mathbb{C}^2\colon |z_2|^2+|z_1|^4<1\}$ (see also \cite[Theorem $1.10$]{NNN24} for the general case). In what follows, we demonstrate that this phenomenon occurs for much larger classes of domains. 

The second aim of this paper is to prove the following theorem, which tells us that the localized squeezing function approaches $1$ along a sequence converging uniformly $\Lambda$-tangentially to a strongly $h$-extendible boundary point (cf. Definition \ref{strongly-h-extendible} and Definition \ref{lambda-tangent}, respectively). 
\begin{theorem}\label{maintheorem2} Let $\Omega$ be a bounded pseudoconvex domain in $\mathbb{C}^{n+1}$ with $C^\infty$-smooth boundary. Let $\xi_0\in\partial \Omega$ be strongly $h$-extendible with Catlin's finite multitype $(2m_1,$ $ \ldots,2m_n, 1)$ and let $\Lambda=(1/2m_1, \ldots, 1/2m_n)$ (see Definition \ref{strongly-h-extendible}). If $\{\eta_j\}\subset \Omega$ is a sequence converging uniformly $\Lambda$-tangentially to $\xi_0 \in \partial \Omega$ (see Definition \ref{lambda-tangent}), then $\sigma_{\Omega\cap U_0}(\eta_j)\to 1$ as $j\to\infty$, where $U_0$ is a sufficiently small  neighborhood of $\xi_0$.
\end{theorem}

The Pinchuk scaling method is useful for strongly pseudoconvex domains in $\mathbb{C}^{n+1}$ (cf. \cite{Wo77, Ro79, Pi91, Ef95}) and pseudoconvex domains of finite type in $\mathbb{C}^{2}$ (cf. \cite{BP89, Be03, Be06}). However, for pseudoconvex domains of finite type in higher dimensions, the Pinchuk scaling method does not work in general. Thus, we need the assumption that the domains are strongly $h$-extendible to apply our scaling techniques effectively. More specifically, the uniform $\Lambda$-tangential convergence of $\{\eta_j\}$ plays a substantial role in proving Theorem \ref{maintheorem2}. Let $\eta_j'$ be the projection of $\eta_j$ onto $\partial \Omega$ along the $\mathrm{Re}(w)$ direction. Then $\partial \Omega$ is strongly pseudoconvex at $\eta_j'$ for all sufficiently large $j$. Therefore, this allows us to choose a suitable scaling sequence so that our model is an analytic ellipsoid that is biholomorphically equivalent to $\mathbb{B}^{n+1}$, and hence Theorem \ref{maintheorem2} follows.

Now we turn to bounded pseudoconvex domains in $\mathbb{C}^2$. First of all, let us consider the following problem, first posed in \cite[Problem $4.3$]{FW18}.
\begin{problem}\label{pb2} Let $\Omega\subset \mathbb{C}^2$ be a bounded pseudoconvex domain of class $\mathcal{C}^\infty$. Is $\sigma_{\Omega}$ bounded away from zero?
\end{problem}
It is noted that if the answer to this question is affirmative, then the domain $\Omega$ is called holomorphic homogeneous regular (cf. \cite{DGZ16}). The main results concerning this property are due to \cite{Ye09, DGZ16, KZ16, NA17}. Examples of holomorphic homogeneous regular domains in all dimensions include bounded strictly pseudoconvex domains of class $\mathcal{C}^2$ (cf. \cite{DGZ16}), bounded convex domains (cf. \cite{Ye09, KZ16}), and bounded $\mathbb{C}$-convex domains (cf. \cite{NA17}). However, the answer to Problem \ref{pb2} is negative in general for domains in higher dimensions (cf. \cite{FR18}).

For a bounded pseudoconvex domain of D'Angelo finite type $\Omega$ in $\mathbb{C}^2$ and a boundary point $\xi_0\in \partial \Omega$, following the proofs given in \cite{Be03} (or in \cite{BP89} for the real-analytic boundary), for each sequence $\{\eta_j\}\subset \Omega$ that converges to $\xi_0$, there exists a scaling sequence $\{f_j\}\subset\mathrm{Aut}(\mathbb{C}^2)$ such that $f_j(\eta_j)$ converges to $(0,-1)$ and $f_j(\Omega)$ converges normally to a model 
$$
M_P=\{(z,w)\in \mathbb{C}^2\colon \mathrm{Re}(w)+P(z)<0\},
$$
where $P$ is a subharmonic polynomial of degree $\leq 2m$, with $2m$ being the type of $\Omega$ at $\xi_0$, without harmonic terms. 

Let us emphasize that the local model $M_P$ depends deeply on the boundary behavior of $\{\eta_j\}$, i.e., the boundary behavior of $\{\eta_j\}$ suggests some choice of a scaling sequence $\{f_j\}$. If $M_P$ is biholomorphically equivalent to a bounded domain, then $\sigma_{M_P}(z)$ is well-defined and by the invariance of the squeezing function one sees that $\sigma_{\Omega}(\eta_j)$ is bounded from below by a positive constant. In the case when $\{\eta_j\}$ converges $\big(\dfrac{1}{2m}\big)$-non-tangentially to $\xi_0$ (see Definition $3.4$ in \cite{NN20}), the model $M_P$ is exactly the associated model for $(\Omega, \xi_0)$ (see Definition \ref{def-order}). Then $\sigma_{\Omega}(\eta_j)$ is bounded from below by a positive constant, provided that the associated model is biholomorphically equivalent to a bounded domain. 

The rest of the paper considers the case when $\{\eta_j\}$ accumulates at $\xi_0$ very tangentially to $\partial \Omega$ -- the remaining possibility. For a bounded pseudoconvex domain of D'Angelo finite type $\Omega$ in $\mathbb{C}^2$, the point $\xi_0$ is $h$-extendible ( see Definition \ref{def-order}). In addition, if $\xi_0$ is strongly $h$-extendible, then the squeezing function tends to $1$ along any sequence converging uniformly $\Lambda$-tangentially to $\xi_0$ by Theorem \ref{maintheorem2}. However, without the strongly $h$-extendibility, the notion of spherically $\frac{1}{2m}$-tangential convergence is necessary to determine if the squeezing function approaches $1$ (cf. Definition \ref{spherically-convergence}).

More precisely, the third aim of this paper is to prove the following theorem.
\begin{theorem}\label{maintheorem3}
Let $\Omega$ be a bounded domain in $\mathbb{C}^2$ and $\xi_0\in \partial \Omega$. Suppose that $\partial \Omega$ is $\mathcal{C}^\infty$-smooth, pseudoconvex and of D'Angelo finite type near $\xi_0$. If $\{\eta_j\}\subset \Omega$ is a sequence converging spherically $\frac{1}{2m}$-tangentially to $\xi_0 \in \partial \Omega$, then $\sigma_{\Omega\cap U_0}(\eta_j)\to 1$ as $j\to\infty$, where $2m$ is the type of $\partial \Omega$ at $\xi_0$, where $U_0$ is a sufficiently small neighborhood of $\xi_0$.
\end{theorem}

Altogether, the results of this paper represent a significant step toward solving Problem \ref{pb2}. The remaining cases to resolve are: (i) removing condition (a) in Definitions \ref{lambda-tangent} and \ref{spherically-convergence}, and (ii) the case where the sequence $\{\eta_j\}$ converges non-spherically (see Example \ref{rmk5.2}). Condition (a) is necessary due to technical restrictions imposed by the scaling method employed in this paper. However, if $\eta_j'$ is chosen as the orthogonal projection of $\eta_j$ onto the boundary, an alternative scaling technique may be applied (see Example \ref{Ex5.1}).

The organization of this paper is as follows. In Section \ref{technical-section}, we recall basic definitions and results needed later. In Section \ref{stronglypsc}, we verify the geometry of strongly pseudoconvex hypersurfaces and prove Theorem \ref{maintheorem1}. In Section \ref{S-h-extendible}, we recall the notion of strongly $\Lambda$-tangential convergence and prove Theorem \ref{maintheorem2}. Finally, the proof of Theorem \ref{maintheorem3} is given in Section \ref{S5}.
\section{Preliminaries}\label{technical-section}

\subsection{Normal convergence} 
We recall the following definition (see \cite{GK87, Kr21}, or \cite{DN09}). 
\begin{define} Let $\{D_j\}_{j=1}^\infty$ be a sequence of domains in $\mathbb C^n$. The sequence $\{D_j\}_{j=1}^\infty$ is said to \emph{converge normally} to a domain $D_0\subset \mathbb C^n$ if the following two conditions hold:
	\begin{enumerate}
		\item[(i)] If a compact set $K$ is contained in the interior (i.e., the largest open subset) of  $\displaystyle\bigcap_{j\geq m} \Omega_j$ for some positive integer $m$, then $K\subset D_0$.
		\item[(ii)] If a compact subset $K'\subset D_0$, then there exists a constant $m>0$ such that $\displaystyle K'\subset \bigcap_{j\geq m} D_j$.
	\end{enumerate}  
Furthermore, when a sequence of  map $f_j\colon D_j\to\mathbb C^m$ converges uniformly on compact sets (also known as uniformly on compacta) to a map $\varphi_j\colon D\to\mathbb C^m$ then we shall say that $\varphi_j$ \emph{converges normally} to $\varphi$. 
\end{define}

\subsection{Catlin's multitype} 
For the convenience of the exposition, let us recall \emph{Catlin's multitype}~(for more details, we refer to \cite{Cat84, Yu95} and the references therein). Let $\Omega$ be a domain in $\mathbb C^n$ and $\rho$ be a defining function for $\Omega$ near $p\in \partial\Omega$. Let us denote by $\Gamma^n$ the set of all $n$-tuples of numbers $\mu=(\mu_1,\ldots,\mu_n)$ such that
\begin{itemize}
\item[(i)] $1\leq \mu_1\leq \cdots\leq\mu_n\leq +\infty$;
\item[(ii)] For each $j$, either $\mu_j=+\infty$ or there is a set of non-negative integers $k_1,\ldots,k_j$ with $k_j>0$ such that
\[
\sum_{s=1}^j \frac{k_s}{\mu_s}=1.
\]
\end{itemize}

A weight $\mu\in \Gamma^n$ is called \emph{distinguished} if there exist holomorphic coordinates $(z_1,\ldots,z_n)$ about $p$ with $p$ maps to the origin such that 
\[
D^\alpha \overline{D}^\beta \rho(p)=0~\text{whenever}~\sum_{i=1}^n\frac{\alpha_i+\beta_i}{\mu_i}<1.
\]
Here and in what follows, $D^\alpha$ and $\overline{D}^\beta$ denote the partial differential operators
\[
\frac{\partial^{|\alpha|}}{\partial z_1^{\alpha_1}\cdots \partial z_n^{\alpha_n }}~\text{and}~\frac{\partial^{|\beta|}}{\partial \bar z_1^{\beta_1}\cdots \partial \bar z_n^{\beta_n }},
\]
respectively. 
\begin{define}
The \emph{multitype} $\mathcal{M}(z_0)$ is defined to be the smallest weight $\mathcal{M}=
(m_1,\ldots,m_n)$ in $\Gamma^n$~(smallest in the lexicographic sense) such that $\mathcal{M}\geq \mu$ for every distinguished weight $\mu$.
\end{define}
 \subsection{The $h$-extendibility }
In what follows, we call a multiindex $(\lambda_1,\lambda_2, \ldots, \lambda_n)$ a \emph{multiweight} if $1\geq \lambda_1\geq \cdots\geq \lambda_n$. Now we recall the following definitions (cf. \cite{Yu94, Yu95}).  
\begin{define} \label{def-28}Let $f(z)$ be a function on $\mathbb  C^n$ and let $\Lambda=(\lambda_1,\lambda_2, \ldots, \lambda_n)$ be a multiweight. For any real number $t\geq 0$, set
$$
\pi_t(z)=(t^{\lambda_1}z_1,t^{\lambda_2}z_2, \ldots,t^{\lambda_n}z_n).
$$ 
We say that $f$ is \emph{$\Lambda$-homogeneous with weight $\alpha$} if $f(\pi_t(z))=t^\alpha f(z)$ for every $t\geq 0$ and $z\in \mathbb C^n$. In case $\alpha=1$, then $f$ is simply called \emph{$\Lambda$-homogeneous}.
\end{define} 

For a multiweight $\Lambda$, the following function
$$
\sigma(z)=\sigma_\Lambda(z):=\sum_{j=1}^n |z_j|^{1/\lambda_j}
$$
 is $\Lambda$-homogeneous. Moreover, for a multiweight $\Lambda$ and a real-valued $\Lambda$-homogeneous function $P$, we define a homogeneous model $D_{\Lambda,P}$ as follows:
 $$
D_{\Lambda,P}=\left\{ (z,w)\in \mathbb C^n\times \mathbb C\colon \mathrm{Re}(w)+ P(z)<0 \right\}.
 $$
 \begin{define} Let $D_{\Lambda,P}$ be a homogeneous model. Then $D_{\Lambda,P}$ is called \emph{$h$-extendible} if there exists a $\Lambda$-homogeneous $\mathcal{C}^1$ function $a(z)$ on $\mathbb C^n\setminus\{0\}$ satisfying the following conditions:
 \begin{itemize}
 \item[(i)] $a(z)>0$ whenever $z\ne 0$;
 \item[(ii)] $P(z)-a(z)$ is plurisubharmonic on $\mathbb C^n$. 
  \end{itemize}
We will call $a(z)$ a \emph{bumping function}. 
 \end{define}

By a pointed domain $(\Omega,p)$ in $\mathbb C^{n+1}$ we mean that $\Omega$ is a smooth pseudoconvex domain in $\mathbb C^{n+1}$ with $p\in \partial\Omega$. Let $\rho$ be a local defining function for $\Omega$ near $p$ and let the multitype $\mathcal{M}(p)=(2m_1,\ldots,2m_n, 1)$ be finite. We note that because of pseudoconvexity, the integers $2m_1,\ldots,2m_n$ are all even. 

 By the definition of multitype, there exist distinguished coordinates $(z,w)=(z_1,\ldots,z_n,w)$ such that $p=(0',0)$ and $\rho(z,w)$ can be expanded near $(0',0)$ as follows:
$$
\rho(z,w)=\mathrm{Re}(w)+P(z)+R(z,w),
$$ 
 where $P$ is a $(1/2m_1,\ldots,1/2m_n)$-homogeneous plurisubharmonic polynomial that contains no pluriharmonic terms, $R$ is smooth and satisfies 
 $$
 |R(z,w)|\leq C \left( |w|+ \sum_{j=1}^n |z_j|^{2m_j} \right)^\gamma,
 $$ 
 for some constant $\gamma>1$ and $C>0$. 
 
 In what follows, we assign weights $\frac{1}{2m_1}, \ldots,\frac{1}{2m_{n}}, 1$ to the variables $z_1,\ldots, z_{n}, w$, respectively and denote by $wt(K):=\sum_{j=1}^{n} \frac{k_j}{2m_j}$ the weight of an $n$-tuple $K=(k_1, \ldots, k_{n})\in \mathbb Z^{n}_{\geq0}$. We note that  $wt(K+L)=wt(K)+wt(L)$ for any $K, L\in \mathbb Z^{n}_{\geq0}$.  In addition, $\lesssim$ and $\gtrsim$ denote inequality up to a positive constant. Moreover, we will use $\approx $ for the combination of $\lesssim$ and $\gtrsim$. 
\begin{define}\label{def-order} We call $M_P=\{(z,w)\in \mathbb C^n\times \mathbb C\colon \mathrm{Re}(w)+P(z)<0\}$ an \emph{associated model} for $(\Omega,p)$. If the pointed domain $(\Omega,p)$ has an $h$-extendible  associated model, we say that $(\Omega,p)$ is \emph{$h$-extendible}. 
\end{define}

Next, we recall the following definition (cf. \cite{Yu95}).
\begin{define} Let $\Lambda=(\lambda_1,\ldots,\lambda_n)$ be a fixed $n$-tuple of positive numbers and $\mu>0$. We denote by $\mathcal{O}(\mu,\Lambda)$ the set of smooth functions $f$ defined near the origin of $\mathbb C^n$ such that
$$
D^\alpha \overline{D}^\beta f(0)=0~\text{whenever}~ \sum_{j=1}^n (\alpha_j+\beta_j)\lambda_j \leq \mu.
$$
In addition, we use $\mathcal{O}(\mu)$ to denote the functions of one variable, defined near the origin of $\mathbb C$,  vanishing to order at least $\mu$ at the origin. 
\end{define}

\section{The boundary behavior of the squeezing function near a strongly pseudoconvex point}\label{stronglypsc}

\subsection{Geometry of strongly pseudoconvex hypersurfaces}

In this subsection, we consider a domain $D$ in $\mathbb{C}^{n+1}$ that is strongly pseudoconvex at $\xi_0 \in \partial D$. After a change of variables, there are the coordinate functions $(z,w)=(z_1,\ldots, z_n,w)$ such that $\xi_0=(0',0)$ and  $\rho(z,w)$, the local defining function for $\Omega$ near $\xi_0$, can be expanded near $(0',0)$ as follows:
\begin{equation*}
\rho(z, w) = \mathrm{Re}(w) + |z|^2+ O(|w||z|+|z|^3+|w|^2).
\end{equation*}

The following proposition plays a central role in the proof of Theorem \ref{maintheorem1}.  Although a proof of this proposition is a minor modification of that given in \cite[Assertion~1]{Ef95}, we shall give a detailed proof for the reader's convenience.
\begin{proposition}\label{prop-1}
Let $D$ be a domain in $\mathbb{C}^{n+1}$ and $\xi_0\in \partial D$. Suppose that $\partial D$ is $\mathcal{C}^2$-smooth near $\xi_0$ and strongly pseudoconvex at $\xi_0 $. Then for each $\eta$, there exists a globally biholomorphic coordinate transformation $\Phi_{\eta}\colon  \mathbb{C}^{n+1} \to \mathbb{C}^{n+1}$ such that the function $\rho(z,w)$ locally defining $D$ has the following form in the new coordinates:
\begin{equation*}
\rho\circ \Phi_{\eta}^{-1}(z,w) =  \mathrm{Re}(w) + |z|^2 +O(|w||z|+|z|^3+|w|^2).
\end{equation*}
\end{proposition}
\begin{proof}
For each $\eta$, we denote by $\eta'$ the point of $\partial D$ closest to $\eta$. Then we first denote by $\varphi_1\in \mathrm{Aut}(\mathbb C^{n+1})$ the composite of the shift $\eta_j \to (0',0)$ and a unitary map taking the complex tangent $T_{\eta_j'}^{\mathbb{C}}(\partial \Omega)$ to the plane $\{v = 0\}$ such that in the new coordinate $(u,v)$ we have
$$
L_{\eta}(\eta)=(0', -\epsilon); L_{\eta}(\eta')=(0',0); 
$$
where $\epsilon$ is the distance from $\eta$ to $\partial \Omega$. Moreover, the tangent to $\partial\Omega$ at $(0', 0)$ is $\{\mathrm{Re}(v) = 0\}$ and the Taylor expansion of the function $\rho\circ \varphi_1^{-1}(u,v)$ locally defining $D$ in a neighbourhood of the origin has the form
\begin{equation*}
\rho\circ \varphi_1^{-1}(u,v) = \mathrm{Re}~ L_\eta(u,v) + \frac{1}{2}H_\eta(u,v) + \mathrm{Re}~ K_\eta(u,v) + o(|u|^2)+o(|v|^2),
\end{equation*}
where
\begin{align*}
L_\eta(u,v) &=2\frac{\partial \rho(\eta')}{\partial v} v+2 \sum_{i=1}^{n} \frac{\partial \rho(\eta')}{\partial u_i} u_i;\\
K_\eta(u,v) &= \sum_{i,j=1}^{n} \frac{\partial^2 \rho(\eta')}{\partial u_i \partial u_j} u_i u_j+\sum_{i=1}^{n} \frac{\partial^2 \rho(\eta')}{\partial u_i \partial v} u_i v + \sum_{i=1}^{n} \frac{\partial^2 \rho(\eta')}{\partial v \partial u_i} v u_i + \frac{\partial^2 \rho(\eta')}{\partial v \partial v} v v;\\
H_\eta(u,v)& = \sum_{i,j=1}^{n} \frac{\partial^2 \rho(\eta')}{\partial u_i \partial \bar{u}_j} u_i \bar{u}_j+\sum_{i=1}^{n} \frac{\partial^2 \rho(\eta')}{\partial u_i \partial \bar v} u_i \bar v + \sum_{i=1}^{n} \frac{\partial^2 \rho(\eta')}{\partial v \partial \bar{u}_i} v \bar{u}_i + \frac{\partial^2 \rho(\eta')}{\partial v \partial \bar v} v \bar v.
\end{align*}

Next, it is standard to perform the change of coordinates $(z,w)=\varphi_2(u,v)$, defined by
\begin{align*}
w&=L_\eta(u,v);\\
z_j&=u_j, \; 1\leq j\leq n.
\end{align*}
Hence, in the coordinates $(z,w)$, the function $\rho\circ \varphi_1^{-1}\circ \varphi_2^{-1}(z,w)$ has the following form
\begin{equation*}
\rho\circ \varphi_1\circ \varphi_2(z,w) =  \mathrm{Re}(w) + \frac{1}{2}H_\eta(z, w) +   \mathrm{Re}~ K_\eta(z,w) + O(|z|^3)+O(|w|^2).
\end{equation*}

Furthermore, since $\partial D$ is strongly pseudoconvex at $\xi_0 $, it follows that $H_\eta(z, 0)$ is a strictly positive Hermitian square form and there exists a linear change of the variables $z_1, \ldots, z_n$, say $P$, that reduces this form to $2|z|^2$. Hence, we define $\varphi_3\in \mathrm{Aut}(\mathbb C^{n+1})$  by
\begin{align*}
u&=P(z);\\
v&=w.
\end{align*}
Then, the defining function $\rho$ can be written in the new coordinates as
\begin{equation*}
\rho\circ \varphi_1^{-1}\circ \varphi_2^{-1}\circ \varphi_3^{-1}(u,v) = \mathrm{Re}(v)+ |u|^2+ \mathrm{Re}~K_\eta(u,v) + O(|u|^3)+O(|v|^2).
\end{equation*}

Finally, we may also perform a change of coordinates $(z,w)=\varphi_4(u,v)$, given by
\begin{align*}
w&=v + K_\eta(u, 0);\\
z&=u.
\end{align*}
The defining function $\rho$ then has the desired expression
$$
\rho\circ \varphi_1^{-1}\circ \varphi_2^{-1}\circ \varphi_3^{-1}\circ \varphi_4^{-1}(z,w)=\mathrm{Re}(w) + |z|^2+O(|w||z|+|z|^3+|w|^2).
$$
Therefore, the required map can be written as $\Phi_\eta =\varphi_4\circ \varphi_3\circ \varphi_2\circ \varphi_1$, and thus the proof is eventually complete.
\end{proof}

\subsection{Proof of Theorem \ref{maintheorem1}}

Let $\Omega$ and $\xi_0\in \partial \Omega $ be as in the statement of Theorem \ref{maintheorem2}. Let $\{\eta_j\}\subset \Omega $ be any sequence converging to  $\xi_0 \in \partial \Omega$ and then we are going to prove that $\sigma_\Omega(\eta_j)\to 1$ as $j\to\infty$. Indeed, we may assume that $\{\eta_j=(\alpha_j,\beta_j)\}_{j\geq 1}\subset U_0^-:=U_0\cap\{\rho<0\}$ for a fixed neighborhood $U_0$ of $\xi_0$ and we associate with each $\eta_j$ a point $\eta_j'\in\partial \Omega$ that is closest to $\eta_j$.

It follows from Proposition \ref{prop-1} that there is a biholomorphism $\Phi_{\eta_j'}$ of $\mathbb C^{n+1}$, $(z,w)=\Phi^{-1}_{\eta_j'}(\tilde z, \tilde w)$ such that $\Phi_{\eta_j'}(\eta_j)=(0',-\epsilon_j), \Phi_{\eta_j'}(\eta_j')=(0',0)$ and
\begin{equation}\label{Eq500}
\rho \circ \Phi_{\eta_j'}^{-1}(\tilde z, \tilde w)  = \mathrm{Re}(\tilde w) + |\tilde z|^2 +O(|\tilde w||\tilde z|+|\tilde z|^3+|\tilde w|^2). 
\end{equation}
 Shrinking $U_0$ if necessary, we may assume that $\Phi_{\eta_j'}(U_0\cap \Omega)$ is contained in the domain $\mathcal{E}:=\big\{(\tilde z, \tilde w)\in \mathbb C^{n+1}\colon A\, \mathrm{Re}(\tilde w) + |\tilde z|^2<B(|\tilde w||\tilde z|+|\tilde w|^2)\big\})$ for all all sufficiently large $j$, where $A,B$ ($A>1$) are positive constants, independent of $j$.

Now let us define $\tau(\eta', \epsilon_j) := \sqrt{\epsilon_j}$ and define an anisotropic dilation $\Lambda_j\in \mathrm{Aut}(\mathbb C^{n+1})$ by
\begin{equation*}
\Lambda_j(z, w) = \Big(\frac{z_1}{\tau(\eta', \epsilon_j)}, \ldots, \frac{z_n}{\tau(\eta', \epsilon_j)}, \frac{w}{\epsilon_j}\Big).
\end{equation*}
Then one sees that $\Delta_j\circ \Phi_{\eta_j'}(\eta_j)=(0',-1),\; \forall j\in \mathbb N_{\geq 1}$. Furthermore, for each $j\in \mathbb N_{\geq 1}$, if we set $\rho_j(z,w)=\epsilon_j^{-1}\rho\circ \Phi_{\eta_j'}^{-1}\circ(\Delta_j)^{-1}(z,w)$, then (\ref{Eq500}) implies that
\begin{equation*}\label{def-jj}
\rho_j(z,w)=\mathrm{Re}(w)+ |z|^2+O(\tau(\eta_j',\epsilon_j)).
\end{equation*}
Therefore, passing to a subsequence if necessary, we may assume that the sequence $\Omega_j:=\Delta_j\circ \Phi_{\eta_j'}(\Omega\cap U_0) $ converges normally to the Siegel half-space
$$
\mathcal{U}_{n+1}:=\{(z,w)\in \mathbb{C}^{n+1} \colon\mathrm{Re}(w) +|z|^2<0\}.
$$
In addition,  the holomorphic map $\Psi $, defined by
\[
(z,w)\mapsto  \Big( \frac{2z_1}{1-w},\ldots, \frac{2z_n}{1-w},  \frac{w+1}{1-w}\Big),
\]
is a biholomorphism from $\mathcal{U}_{n+1}$ onto $\mathbb B^{n+1}$.

Next, let us consider the sequence of biholomorphic map $f_j:=\Psi\circ \Delta_j\circ \Phi_{\eta_j'} \colon \Omega\cap U_0 \to \Psi(\Omega_j)$.  Then, for any $(z,w)\in U_0\cap \Omega$, the point $(z',w'):= \Delta_j\circ \Phi_{\eta_j'}(z,w)\in \Omega_j$ satisfies
\[
\mathrm{Re}(w')\gtrsim -\dfrac{1}{\epsilon_j}; \quad |z_k'|\lesssim \frac{1}{\sqrt{\epsilon_j}},\quad 1\leq k\leq n.
\]
Moreover, for all sufficiently large $j$, we have 
\[
\Omega_j\subset \Delta_j(\mathcal{E}) =\{('z','w')\in \mathbb{C}^{n+1}\colon A\, \mathrm{Re}(w') + |\tilde z|^2<B(\sqrt{\epsilon_j}|w'||z'|+\epsilon_j|w'|^2)\}.
\]
 Consequently, $|z'|^2 \lesssim |\mathrm{Re}(w')|$ for all $(z',w')\in \Omega_j$. In addition, for $(\tilde z,\tilde w)=\Psi(z',w')$ with $(z',w')\in \Omega_j$, we have the estimates
\begin{align*}
|\tilde w+1|=\Big|\dfrac{2}{1-w'}\Big|\lesssim \dfrac{1}{|\mathrm{Re}(w')|},\quad |\tilde z_k|=\Big|\frac{2z'_k}{1-w'}\Big|\lesssim \frac{\sqrt{|\mathrm{Re}(w')|}}{|\mathrm{Re}(w')|-1}\lesssim \frac{1}{\sqrt{|\mathrm{Re}(w')|}}
\end{align*}
for $1\leq k\leq n$. Therefore, for any sufficiently small $\epsilon > 0$, there exist $M>0$ and $j_0 \in \mathbb{N}_{\geq 1}$ such that the image under $\Psi$ of the domain
\[
\Big\{(z',w')\in \Omega_j\colon \mathrm{Re}(w')< -M\Big\}
\]
is contained in $B\big((0',-1), \epsilon/2\big)$ for all $j\geq j_0$. In addition, the sequence $\big\{(z',w')\in \overline{\Omega_j}\colon \mathrm{Re}(w')\geq -M\big\}$ converges to the set $\big\{(z',w')\in \mathcal{U}_{n+1}\colon \mathrm{Re}(w')\geq -M\big\}$ in the Hausdorff sense. Furthermore, we have
\[
f_j(\eta_j) = \Psi\circ \Theta(0', -1) =(0', 0) \quad \text{for all } j \in \mathbb N_{\geq 1}.
\]
This yields that for any sufficiently small $\epsilon > 0$, there exists $j_0 \in \mathbb{N}_{\geq 1}$ such that
\[
B\big((0',0), 1 - \epsilon\big) \subset f_j(U_0\cap \Omega) \subset B\big((0',0), 1 + \epsilon\big), \quad \forall j \geq j_0, 
\]
and hence
\[
\sigma_{U_0\cap \Omega}(\eta_j) \geq \frac{1 - \epsilon}{1 + \epsilon}, \quad \forall j \geq j_0,
\]
Since $\epsilon>0$ is arbitrary, we conclude that $\displaystyle \lim_{j\to \infty}\sigma_{U_0\cap \Omega}(\eta_j)=1$, and  thus the proof of Theorem~\ref{maintheorem1} is complete. \hfill $\Box$

\section{The boundary behavior of the squeezing function near a strongly $h$-extendible point}\label{S-h-extendible}
\subsection{$\Lambda$-tangential convergence}\label{Ss3.1}

Throughout this subsection, let $\Omega$ be a domain in $\mathbb C^{n+1}$ and assume that $\xi_0\in \partial \Omega $ is an $h$-extendible boundary point \cite{Yu95} (or, semiregular point in the terminology of \cite{DH94}). Let $\mathcal{M}(\xi_0)=(2m_1,\ldots,2m_n,1)$ be the finite multitype of  $\partial \Omega$ at $\xi_0$ (see \cite{Cat84}). (Note that because of the pseudoconvexity of $\Omega$, the integers $2m_1,\ldots,2m_n$ are all even.) Let us denote by $\Lambda=(1/2m_1,\ldots,1/2m_n)$. By following the proofs of Lemmas $4.10$, $4.11$ in \cite{Yu95}, after a change of variables there are the coordinate functions $(z,w)=(z_1,\ldots, z_n,w)$ such that $\xi_0=(0',0)$ and  $\rho(z,w)$, the local defining function for $\Omega$ near $\xi_0$, can be expanded near $(0',0)$ as follows:
$$
\rho(z,w)=\mathrm{Re}(w)+ P(z) +R_1(z) + R_2(\mathrm{Im} w)+(\mathrm{Im} w) R(z),
 $$ 
 where $P$ is a $\Lambda$-homogeneous plurisubharmonic polynomial that contains no pluriharmonic monomials, $R_1\in \mathcal{O}(1, \Lambda),R\in \mathcal{O}(1/2, \Lambda) $, and $R_2\in \mathcal{O}(2)$.

We know that a sequence $\{\eta_j\}\subset \Omega$ converges $\Lambda$-nontangentially to $\xi_0$ if $|\mathrm{Im}(\beta_j)|\lesssim |\mathrm{dist}(\eta_j,\partial \Omega)|$ and $|\alpha_{j k}|^{2m_k}\lesssim|\mathrm{dist}(\eta_j,\partial \Omega)|$ for every $1\leq k\leq n$ (cf. \cite{NN20}). Here and in what follows, $\mathrm{dist}(z,\partial\Omega)$ denotes the Euclidean distance from $z$ to $\partial\Omega$.

 The following definition gives us a type of $\Lambda$-tangential convergence. 
\begin{define}[\cite{NNN25}]\label{lambda-tangent}
We say that a sequence $\{\eta_j=(\alpha_j,\beta_j)\}\subset  \Omega$ with $\alpha_j=(\alpha_{j 1},\ldots,\alpha_{j n})$, \emph{converges uniformly $\Lambda$-tangentially to $\xi_0$} if the following conditions hold:
\begin{itemize}
\item[(a)] $|\mathrm{Im}(\beta_j)|\lesssim |\mathrm{dist}(\eta_j,\partial \Omega)|$;
\item[(b)] $|\mathrm{dist}(\eta_j,\partial \Omega)|=o(|\alpha_{jk}|^{2m_k})$ for $1\leq k\leq n$;
\item[(c)] $|\alpha_{j1}|^{2m_1}\approx |\alpha_{j2}|^{2m_2}\approx \cdots\approx |\alpha_{jn}|^{2m_n}$.
\end{itemize}
\end{define}
\begin{remark} 
In the case when the point $\xi_0$ is strongly pseudoconvex, as in Theorem \ref{maintheorem1}, condition (a) is not necessary. However, this condition (a) is necessary due to technical restrictions, such as the scaling method employed in the proofs of Theorem \ref{maintheorem2} and Theorem \ref{maintheorem3} in this section and the next section, respectively.
\end{remark}

Now let us denote by $\displaystyle \sigma(z):=\sum_{k=1}^n |z_k|^{2m_k}$ and recall the following definition.
\begin{define}[\cite{NNN25}]\label{strongly-h-extendible} We say that a boundary point $\xi_0\in \partial \Omega$ is \emph{strongly $h$-extendible} if there exists $\delta>0$ such that $P(z)-\delta \sigma(z)$ is plurisubharmonic, i.e. $dd^c P\geq \delta dd^c \sigma$. 
\end{define}
\begin{remark}\label{remark-strongly-h-extendible} 
Since $dd^c P\gtrsim  dd^c \sigma$, it follows that
\begin{align*}
\sum_{k, l=1}^n \frac{\partial^2P}{\partial z_k\partial \bar z_l} (\alpha) w_j\bar w_l&\gtrsim \sum_{k, l=1}^n \frac{\partial^2\sigma}{\partial z_k\partial \bar z_l} (\alpha) w_j\bar w_l\\
&\gtrsim m_1^2|\alpha_1|^{2m_1-2}|w_1|^2+\cdots+m_n^2|\alpha_n|^{2m_n-2}|w_n|^2
\end{align*}
for all $\alpha,w\in \mathbb C^n$. Consequently, $P$ is strictly plurisubharmonic away from the union of all coordinates axes, i.e. $M_P$ is \emph{homogeneous finite diagonal type} in the sense of  \cite{He92, He16} (or $M_P$ is a \emph{$WB$-domain} in the sense of  \cite{AGK16}).
\end{remark}
\begin{example} \label{ball} Let $\mathcal {E}_{1,2,3}$ be the domain in $\mathbb C^{n+1}$ defined by
$$
\mathcal {E}_{1,2,3}:=\left\{(z_1,z_2,w)\in \mathbb C^{3}\colon \rho(z,w):=\mathrm{Re}(w)+ |z_1|^4+|z_2|^6 <0\right\}.
$$
We note that $\mathcal {E}_{1,2,3}$ is biholomorphically equivalent to the ellipsoid 
$$
\mathcal{D}_{{1,2,3}}:=\left\{(z_1,z_2,w)\in \mathbb C^{3}\colon |w|^2+ |z_1|^4+|z_2|^6<1\right\}
$$
 (cf. \cite{BP95, NNTK19}). Moreover, since $P(z_1,z_2)=|z_1|^4+|z_2|^6=\sigma(z_1,z_2)$ it is obvious that the boundary point $(0,0,0)\in \partial \mathcal{E}_{1,2,3}$ is strongly $h$-extendible.

 Now let us define a sequence $\{\eta_j\}\subset \mathcal {E}_{1,2,3}$ by setting $\eta_j=\big(1/j^{1/4}, 1/j^{1/6},-2/j-1/j^2\big)$ for every $j\in \mathbb N_{\geq 1}$. Then $\rho(\eta_j)=-1/j^2\approx -\mathrm{dist}(\eta_j,\partial\mathcal {E}_{1,2,3})$, and thus $\mathrm{dist}(\eta_j,\partial \mathcal{E}_{1,2,3})=o(\Big|\dfrac{1}{j^{1/4}}\Big|^4)=o(\Big|\dfrac{1}{j^{1/6}}\Big|^6)$. Hence, the sequence $\{\eta_j\}\subset  \mathcal{E}_{1,2,3}$ converges uniformly $\big(\dfrac{1}{4},\dfrac{1}{6}\big)$-tangentially to $(0,0,0)\in \partial \mathcal{E}_{1,2,3}$. \hfill $\Box$
\end{example}

In the sequel, we will assume that $\xi_0\in \partial\Omega$ is a strongly $h$-extendible point and let $\{\epsilon_j\}\subset \mathbb R^+$ be a given sequence. Then we define the sequence $\tau_j=(\tau_{j1},\ldots,\tau_{jn})$, associated to $\{\epsilon_j\}$, as follows:
\begin{equation*}
\tau_{jk}:=|\alpha_{jk}|.\Big(\dfrac{\epsilon_j}{|\alpha_{jk}|^{2m_k}}\Big)^{1/2},\; \forall j\geq 1, 1\leq k\leq n.
\end{equation*}
A simple calculation shows that $\tau_{jk}^{2m_k}=\epsilon_j.\Big(\frac{\epsilon_j}{|\alpha_{jk}|^{2m_k}}\Big)^{m_k-1}\lesssim \epsilon_j$. Hence, we get the following estimates
\begin{align}\label{tau-estimate}
\epsilon_j^{1/2}\lesssim \tau_{jk}\lesssim \epsilon_j^{1/2m_k}.
\end{align}

In order to prove Theorem \ref{maintheorem2}, we recall the following lemma (see a proof in \cite{NNN25}). 
 \begin{lemma}[\cite{NNN25}]\label{Cn-spherical-convergence} If $P(z)-\delta \sigma(z)$ is plurisubharmonic for some $\delta>0$, then
\begin{align*}
\epsilon_j^{-1}\sum_{k, l=1}^n \frac{\partial^2P}{\partial z_k\partial \bar z_l} (\alpha_j)\tau_{jk}\tau_{jl} w_k\bar w_l\gtrsim m_1^2|w_1|^2+\cdots+m_n^2|w_n|^2.
\end{align*}
\end{lemma}

\subsection{Proof of Theorem \ref{maintheorem2}}
Let $\Omega$ and $\xi_0\in \partial \Omega $ be as in the statement of Theorem \ref{maintheorem2}. Let $\mathcal{M}(\xi_0)=(2m_1,\ldots,2m_n,1)$ be the finite multitype of $\Omega$ at $\xi_0$ and denote by $\Lambda=(1/2m_1,\ldots,1/2m_n)$. As in Subsection \ref{Ss3.1}, one can find local coordinates $(z, w)=(z_1,\ldots,z_n,w)$ near $\xi_0$ such that $\xi_0=(0',0)$ and the local defining function $\rho(z,w)$ for $\Omega$ can be expanded near $(0',0)$ as follows:
$$
\rho(z,w)=\mathrm{Re}(w)+ P(z) +R_1(z) + R_2(\mathrm{Im} w)+(\mathrm{Im} w) R(z),
$$ 
 where $P$ is a $\Lambda$-homogeneous plurisubharmonic polynomial that contains no pluriharmonic monomials, $R_1\in \mathcal{O}(1, \Lambda),R\in \mathcal{O}(1/2, \Lambda) $, and $R_2\in \mathcal{O}(2)$.

By hypothesis of Theorem \ref{maintheorem2}, the sequence $\{\eta_j\}$ converges uniformly $\Lambda$-tangentially to $\xi_0$. If we write $\eta_j=(\alpha_j,\beta_j)=(\alpha_{j1},\ldots,\alpha_{jn},\beta_j)$, then we have
\begin{itemize}
\item[(a)] $|\mathrm{Im}(\beta_j)|\lesssim |\mathrm{dist}(\eta_j,\partial \Omega)|$;
\item[(b)] $|\mathrm{dist}(\eta_j,\partial \Omega)|=o(|\alpha_{jk}|^{2m_k})$ for $1\leq k\leq n$;
\item[(c)] $|\alpha_{j1}|^{2m_1}\approx |\alpha_{j2}|^{2m_2}\approx \cdots\approx |\alpha_{jn}|^{2m_n}$.
\end{itemize}
Let us fix a small neighborhood $U_0$ of  the origin. Then, without loss of generality we may assume that $\{\eta_j=(\alpha_j,\beta_j)\}\subset U_0^-:=U_0\cap\{\rho<0\}$ and one associates with a sequence of points $\eta_j'=(\alpha_{j}, a_j +\epsilon_j+i b_j)$, where $\epsilon_j>0$ and $\beta_j=a_j+i b_j$, such that $\eta_j'=(\alpha_j,\beta_j')$ with $\beta_j'= a_j +\epsilon_j+i b_j$ is in the hypersurface $\{\rho=0\}$ for every $j\in\mathbb N_{\geq 1}$. Let us note that $\epsilon_j\approx \mathrm{dist}(\eta_j,\partial \Omega)$ .

We now proceed with the scaling method. To do this, as in the proof of Theorem $1.1$ in \cite{NNN25} we make several changes of coordinates as follows. We first define the sequences of  translations $L_{\eta_j'}\colon \mathbb C^{n+1}\to\mathbb C^{n+1}$, defined by
$$
(\tilde z,\tilde w)=L_{\eta_j'}(z,w):=(z,w)-\eta_j'=(z-\alpha_j,w-\beta_j'),
$$
and we consider the sequence $\{Q_j\}$ of  automorphisms of $\mathbb C^{n+1}$, given by 
\[
\begin{cases}
w:= \tilde{w}+(R_2'(b_j) + R(\alpha_j)) i\tilde w+2\sum\limits_{1\leq |p|\leq 2} \frac{D^pP}{p!} (\alpha_j)(\tilde z)^p+2\sum\limits_{1\leq |p|\leq 2} \frac{D^p R_1}{p!}(\alpha_j)  (\tilde z)^p\\
\hskip 1cm+b_j\sum\limits_{1\leq |p|\leq 2} \frac{D^p R}{p!}(\alpha_j)(\tilde z)^p ;\\
z_k:=\tilde{z}_k,\, k=1,\ldots,n.
\end{cases}
\]

We finally define an anisotropic dilation $\Delta_j\colon  \mathbb C^{n+1}\to\mathbb C^{n+1}$ by settings:
\begin{equation*}\label{dilationj}
\Delta_j(z,w):=\Delta_{\eta_j}^{\epsilon_j} (z_1,\ldots,z_n, w)=\Big(\frac{z_1}{\tau_{j1}},\ldots,\frac{z_n}{\tau_{jn}}, \frac{w}{\epsilon_j}\Big),
\end{equation*}
where 
\begin{equation*}
\tau_{jk}:=|\alpha_{jk}|.\Big(\dfrac{\epsilon_j}{|\alpha_{jk}|^{2m_k}}\Big)^{1/2},\; 1\leq k\leq n.
\end{equation*}

As a result, the composition $T_j:=\Delta_j\circ Q_j\circ L_{\eta_j'}\in \mathrm{Aut}(\mathbb C^{n+1})$ satisfies that $T_j(\eta_j')=(0',0)$ and $T_j(\eta_j)=(0',-1-i(R_2'(b_j) + R(\alpha_j)))\to (0',-1)$ as $j\to\infty$. Moreover, the hypersurface $T_j(\{\rho=0\}) $ is now defined by an equation of the form
\begin{align}\label{taylor-defining-function}
\begin{split}
&\epsilon_j^{-1}\rho\left (T_j^{-1}(\tilde z,\tilde w)\right)\\
&= \mathrm{Re} (\tilde w)+\epsilon_j^{-1}o(\epsilon_j|\mathrm{Im}(\tilde w)|)+\frac{1}{2}\sum_{k,l=1}^n \frac{\partial^2 P}{\partial \tilde z_k\partial\overline{\tilde z_l}}(\alpha_j) \epsilon_j^{-1}\tau_{jk}\tau_{jl} \tilde z_k\overline{\tilde z_l}\\
&+\frac{1}{2}\sum_{k,l=1}^n \frac{\partial^2 R_1}{\partial \tilde z_k\partial\overline{\tilde z_l}}(\alpha_j) \epsilon_j^{-1}\tau_{jk}\tau_{jl} \tilde z_k\overline{\tilde z_l}+ \frac{\epsilon_j^{-1}b_j}{2}\sum_{k,l=1}^n \frac{\partial^2 R}{\partial \tilde z_k\partial\overline{\tilde z_l}}(\alpha_j) \tau_{jk}\tau_{jl} \tilde z_k\overline{\tilde z_l}+\cdots=0,
\end{split}
\end{align}
where the dots denote remainder terms. 

Thanks to the fact that $\{\eta_j\}$ converges uniformly $\Lambda$-tangentially to  $\xi_0=(0',0)$, the authors \cite{NNN25} proved that, after taking a subsequence if necessary, the sequence of defining functions given in (\ref{taylor-defining-function}) converges uniformly on compacta of $\mathbb C^{n+1}$ to $\hat\rho(\tilde z,\tilde w):=\mathrm{Re}(\tilde w)+H(\tilde z)$, where
$$
H(\tilde z)=\sum_{k,l=1}^n a_{kl}  \tilde z_k\overline{\tilde z_l}
$$
with coefficients $a_{kl}$ given by
$$
a_{kl}:=\frac{1}{2}\lim_{j\to\infty}\frac{\partial^2 P}{\partial \tilde z_k\partial\overline{\tilde z_l}}(\alpha_j) \epsilon_j^{-1}\tau_{jk}\tau_{jl}, 1\leq k,l\leq n.
$$

One notes that $M_H$ is also the limit of a sequence of the pseudoconvex domains $T_j(\Omega \cap U_0) $. Hence, $M_H$ is also pseudoconvex, and thus $H$ is plurisubharmonic. In addition, it follows directly from Lemma \ref{Cn-spherical-convergence} that $H$ is positive definite. Therefore, there exits a biholomorphism $\Theta \colon M_{H}\to \mathcal{U}_{n+1}$, where $\mathcal{U}_{n+1}$ is the Siegel half-space, given by 
$$
\mathcal{U}_{n+1}:=\{(z,w)\in \mathbb{C}^{n+1} \colon\mathrm{Re}(w) +|z_1|^2+|z_2|^2+\cdots+|z_n|^2<0\}.
$$
It is important to note that the map $\Theta$ is chosen as a composition of a dilation and a unitary transformation (in the variables $(\tilde z_l,\ldots, \tilde z_n)$) that diagonalizes $H(\tilde z)$ (see the proof of Prop. 2 in \cite{Gra75}). In addition,  the holomorphic map $\Psi $ defined by
\[
(z,w)\mapsto  \Big( \frac{2z_1}{1-w},\ldots, \frac{2z_n}{1-w},  \frac{w+1}{1-w}\Big),
\]
is a biholomorphism from $\mathcal{U}_{n+1}$ onto $\mathbb B^{n+1}$.

Now let us consider the sequence of biholomorphic maps $f_j:=\Psi\circ \Theta\circ\Delta_j\circ Q_j\circ L_{\eta_j'} \colon \Omega \to  f_j(\Omega)=\Psi\circ \Theta(\Omega_j)$. Then the sequences $f_j(\Omega \cap U_0)$ and $f_j(\partial \Omega \cap U_0)$ converge normally to $\mathbb{B}^{n+1}$ and $\partial \mathbb{B}^{n+1}$, respectively. Let us denote $\Omega_j:=\Theta\circ\Delta_j\circ Q_j\circ L_{\eta_j'}(\Omega \cap U_0)$ for all $j\in \mathbb{N}_{\geq 1}$. Thus, for any $(z,w)\in U_0\cap \Omega$, the point $(z',w'):= \Theta\circ\Delta_j\circ Q_j\circ L_{\eta_j'}(z,w)\in \Omega_j$ satisfies
\[
\mathrm{Re}(w')\gtrsim -\dfrac{1}{\epsilon_j}; \quad |z_k'|\lesssim \frac{1}{\tau_{jk}},\quad 1\leq k\leq n.
\]
As in the proof of Theorem \ref{maintheorem1}, by \cite[Lemma 3.2 and Lemma 3.3]{NNN25} and the equation \eqref{taylor-defining-function}, shrinking $U_0$ if necessary we may assume that
$$
|z'|^2 \lesssim |\mathrm{Re}(w')|
$$ 
for all $(z',w')\in\Omega_j$, where the constant is independent of $j$. Moreover, for $(\tilde z,\tilde w)=\Psi(z',w')$ with $(z',w')\in \Omega_j$, we have the estimates
\begin{align*}
|\tilde w+1|=\Big|\dfrac{2}{1-w'}\Big|\lesssim \dfrac{1}{|\mathrm{Re}(w')|},\quad |\tilde z_k|=\Big|\frac{2z'_k}{1-w'}\Big|\lesssim \frac{\sqrt{|\mathrm{Re}(w')|}}{|\mathrm{Re}(w')|-1}\lesssim \frac{1}{\sqrt{|\mathrm{Re}(w')|}}
\end{align*}
for $1\leq k\leq n$. Therefore, for any sufficiently small $\epsilon > 0$, there exist $M>0$ and $j_0 \in \mathbb{N}_{\geq 1}$ such that the image under $\Psi$ of the domain
\[
\Big\{(z',w')\in \Omega_j\colon \mathrm{Re}(w')< -M\Big\}
\]
is contained in $B\big((0',-1), \epsilon/2\big)$ for all $j\geq j_0$. In addition, the sequence $\big\{(z',w')\in \Omega_j\colon \mathrm{Re}(w')\geq -M\big\}$ converges to the set $\big\{(z',w')\in \mathcal{U}_{n+1}\colon \mathrm{Re}(w')\geq -M\big\}$ in the Hausdorff sense. Furthermore, since $\Theta(0',-1) =(0',-1)$ and $\Psi(0',-1)=(0',0)$, we have 
\[
f_j(\eta_j) = \Psi\circ \Theta(0', -1 - i(R_2'(b_j) + R(\alpha_j))) \to (0', 0) \quad \text{as } j \to \infty.
\]
This yields that for any sufficiently small $\epsilon > 0$, there exists $j_0 \in \mathbb{N}_{\geq 1}$ such that
\[
B\big((0',0), 1 - \epsilon\big) \subset F_j(U_0\cap \Omega) \subset B\big((0',0), 1 + \epsilon\big), \quad \forall j \geq j_0,
\]
where $F_j(\cdot) := f_j(\cdot) - f_j(\eta_j)$ for all $j \geq j_0$. Since $F_j(\eta_j) = 0$, it follows that
\[
\sigma_{\Omega\cap U_0}(\eta_j) \geq \frac{1 - \epsilon}{1 + \epsilon}, \quad \forall j \geq j_0.
\]
Since $\epsilon>0$ is arbitrary, we conclude that $\displaystyle \lim_{j\to \infty}\sigma_{\Omega\cap U_0}(\eta_j)=1$, and thus the proof of Theorem~\ref{maintheorem2} is complete.
\hfill $\Box$

\begin{example} \label{ex3.1}  Denote by $E_{1,2,4}$ the domain in $\mathbb C^3$, given by
$$
E_{1,2,4}:=\{(z_1,z_2, w)\in \mathbb C^3\colon \mathrm{Re}(w)+|z_1|^4+|z_1|^2|z_2|^4+|z_2|^8<0\}.
$$
Denote by $P(z)=|z_1|^4+|z_1|^2|z_2|^4+|z_2|^8$ and $\sigma(z)=|z_1|^4+|z_2|^8$. Then a computation shows that 
\begin{align*}
dd^c P(z)&=(4|z_1|^2+|z_2|^4) d z_1 d\bar z_1+2 \bar z_1 z_2 |z_2|^2d z_1 d\bar z_2+ 2z_1\bar z_2|z_2|^2 d \bar z_1 d z_2\\
&+(16|z_2|^6+4|z_1|^2|z_2|^2) d z_2 d\bar z_2\\
&=4|z_1|^2 d z_1 d\bar z_1+16|z_2|^6 d z_2 d\bar z_2+ |z_2|^2 |z_2  d\bar z_1+2 \bar z_1 d z_2|^2\\
&\geq dd^c\sigma(z).
\end{align*}
Therefore, the origin is strongly $h$-extendible with multitype $(4,8,1)$ and thus the weight $\Lambda$ is now given by $\displaystyle \Lambda:=(\frac{1}{4},\frac{1}{8})$.

Although $E_{1,2,4}$ is unbounded, but it is biholomorphically equivalent to the bounded domain
$$
\big\{(z_1,z_2, w)\in \mathbb C^3\colon |w|^2+|z_1|^4+|z_1|^2|z_2|^4+|z_2|^8<1\big\}
$$
via the following biholomorphism 
\[
(z_1, z_2,w)\mapsto \Big( \frac{z_1}{(1+w)^{1/2}},  \frac{z_2}{(1+w)^{1/4}},  \frac{w-1}{1+w}\Big).
\]
Therefore, the squeezing function of $E_{1,2,4}$, denoted by $\sigma_{E_{1,2,4}}$, is well-defined. \hfill$\Box$
\end{example}

To complete this section, we shall prove the following proposition using a variant of the scaling method.
\begin{proposition}\label{pro2} Let $\Omega$ be a bounded domain in $\mathbb C^3$ and $(0,0,0)\in \partial\Omega$. Suppose that the defining function $\rho$ for $\Omega$ near  $(0,0,0)$ given by
$$
\rho(z_1,z_2,w)=\mathrm{Re}(w)+|z_1|^4+|z_1|^2|z_2|^4+|z_2|^8.
$$
Then we have
$$
\liminf_{j\to\infty}\sigma_{\Omega\cap U_0}(\eta_j)>0,
$$
where $\displaystyle\eta_j=\big(\frac{1}{j^{1/4}},\frac{1}{j^{3/8}},-\frac{1}{j}-\frac{2}{j^2}-\frac{1}{j^3}\big)\in\Omega, \; \forall j\in \mathbb N_{\geq 1}$, and $U_0$ is a sufficiently small neighborhood of $(0,0,0)$.

\end{proposition}
\begin{proof}
As in Example \ref{ex3.1}, the origin is strongly $h$-extendible with multitype $(4,8,1)$ and thus we denote by $\displaystyle \Lambda:=(\frac{1}{4},\frac{1}{8})$.
Now we consider the sequence $\big\{\eta_j:=(\frac{1}{j^{1/4}},\frac{1}{j^{3/8}},-\frac{1}{j}-\frac{2}{j^2}-\frac{1}{j^3})\big\}$ that converges $\Lambda$-tangentially but not uniformly to $(0,0,0)$.

Although we cannot apply the scaling method given in the proof of Theorem \ref{maintheorem2}, an alternative scaling can be introduced as follows. Indeed, let $\rho(z_1,z_2,w)=\mathrm{Re}(w)+|z_1|^4+|z_1|^2|z_2|^4+|z_2|^8$ and let $\eta_j=(\frac{1}{j^{1/4}},\frac{1}{j^{3/8}},-\frac{1}{j}-\frac{2}{j^2}-\frac{1}{j^3})$ for every $j\in \mathbb N_{\geq 1}$. Then $\eta_j=(\frac{1}{j^{1/4}},\frac{1}{j^{3/8}},-\frac{1}{j}-\frac{1}{j^2}-\frac{1}{j^3})\in \partial \Omega$ for every $j\in \mathbb N_{\geq 1}$. It is noted that $\rho(\eta_j)=-\frac{1}{j^2}\approx -\mathrm{dist}(\eta_j,\partial \Omega)$ and let us set
$\epsilon_j=|\rho(\eta_j)|=\frac{1}{j^2}$. 

We first consider a change of variables $(\tilde z,\tilde w):=L_j(z,w)$, i.e.,
\[
\begin{cases}
w=\tilde w;\\
\displaystyle z_1-\frac{1}{j^{1/4}}= \tilde{z}_1;\\
\displaystyle z_2-\frac{1}{j^{3/8}}=\tilde{z}_2.
\end{cases}
\]
Then, a direct calculation shows that
\begin{equation*}
\begin{split}
&\rho\circ L_j^{-1} (\tilde w,\tilde z_1,\tilde z_2)=
\mathrm{Re}(\tilde w)+ |\dfrac{1}{j^{1/4}}+\tilde z_1|^4+ |\dfrac{1}{j^{1/4}}+\tilde z_1|^2 |\frac{1}{j^{3/8}}+\tilde z_2|^4+|\frac{1}{j^{3/8}}+\tilde z_2|^8 \\
                      &= \mathrm{Re}(\tilde w)+\frac{1}{j}+\frac{4}{j^{3/4}}\mathrm{Re}(\tilde z_1) +\frac{2}{j^{1/2}}|\tilde z_1|^2+\frac{1}{j^{1/2}}(2\text{Re}(\tilde z_1))^2
+ \frac{4}{j^{1/4}} |\tilde z_1|^2 \mathrm{Re}(\tilde z_1)+ |\tilde z_1|^4\\
&+\Big( \frac{1}{j^{1/2}}+\dfrac{2}{j^{1/4}}\mathrm{Re}(\tilde z_1)+|\tilde z_1|^2\Big) \times\\
 &\Big(\frac{1}{j^{3/2}}+\frac{4}{j^{9/8}}\mathrm{Re}(\tilde z_2) +\frac{2}{j^{3/4}}|\tilde z_2|^2+\frac{1}{j^{3/4}}(2\text{Re}(\tilde z_2))^2
+ \frac{4}{j^{3/8}} |\tilde z_2|^2 \mathrm{Re}(\tilde z_2)+ |\tilde z_2|^4\Big)+|\frac{1}{j^{3/8}}+\tilde z_2|^8 .
\end{split}
\end{equation*}
To define an anisotropic dilation,  let us denote by
 $\tau_{1j}:=\tau_1(\eta_j)=\frac{1}{2 j^{3/4}}$ and $\tau_{2j}:=\tau_2(\eta_j)=\frac{1}{ j^{3/8}}$ for all $j\in \mathbb N_{\geq 1}$. Now we introduce a sequence of  polynomial automorphisms $\phi_{{\eta}_j}$ of $\mathbb C^3$ ($j\in \mathbb N_{\geq 1}$), given by
\begin{equation*}
\begin{split}
&\phi_{{\eta}_j} ^{-1}(\tilde z_1,\tilde z_2,\tilde w)\\
&= \Big (\dfrac{1}{j^{1/4}}+\tau_{1j} \tilde z_1, \,\dfrac{1}{j^{3/8}}+ \tau_{2j} \tilde z_2, \,-\frac{1}{j}-\frac{1}{j^2}-\frac{1}{j^3}+\epsilon_j \tilde w- \frac{4}{j^{3/4}}\tau_{1j} \tilde z_1-\frac{2}{j^{1/2}}(\tau_{1j})^2 \tilde z_1^2\Big).
\end{split}
\end{equation*}
Therefore, for each $j\in\mathbb N_{\geq 1}$ the hypersurface $\phi_{\eta_j}(\{\rho=0\}) $ is then defined by

\begin{equation*}
\begin{split}
&\epsilon_j^{-1}\rho\circ \phi_{{\eta}_j} ^{-1}(\tilde z_1,\tilde z_2,\tilde w)\\
&= \epsilon_j^{-1}\rho \Big (\dfrac{1}{j^{1/4}}+\tau_{1j} \tilde z_1, \,\dfrac{1}{j^{3/8}}+ \tau_{2j} \tilde z_2, \,-\frac{1}{j}-\frac{1}{j^2}-\frac{1}{j^3}+\epsilon_j \tilde w- \frac{4}{j^{3/4}}\tau_{1j} \tilde z_1-\frac{2}{j^{1/2}}(\tau_{1j})^2 \tilde z_1^2\Big)\\
&=\mathrm{Re}(\tilde w) +|\tilde z_1|^2+ \frac{1}{16j}|\tilde z_1|^4
+ \frac{1}{2j^{1/4}}|\tilde z_1|^2\mathrm{Re}(\tilde z_1)+ \left(|\tilde z_2+1|^4-1\right)+O(\frac{1}{j^{1/2}})=0.
\end{split}
\end{equation*} 
This yields that the sequence of domains $\Omega_j:=\phi_{{\eta}_j}(\Omega\cap U_0) $, where $U_0$ is a sufficiently small neighborhood of $(0,0,0)$, converges normally to the following model
$$
M_{1,2}:=\left \{(\tilde z_1,\tilde  z_2, \tilde  w)\in \mathbb C^3\colon \mathrm{Re}(\tilde  w)+|\tilde  z_1|^2+ \left(|\tilde z_2+1|^4-1\right)<0\right\}.
$$

Next, one observes that $\phi_{{\eta}_j}(\eta_j')=(0,0,0)\in \partial \Omega_j$ and $\phi_{{\eta}_j}(\eta_j)=(0,0,-1)\in \Omega_j$ for all $j\in \mathbb N_{\geq 1}$. Let us define a biholomorphic map $\Theta$, given by
\[ 
w= \tilde w-1,  z_1=\tilde z_1, z_2=\tilde z_2+1, 
\]
maps $M_{1,2}$ onto the following domain
$$
\mathcal{E}_{1,1,2}=\left \{(z_1,z_2, w)\in \mathbb C^3\colon \mathrm{Re}(w)+|z_1|^2+ |z_2|^4<0\right\}.
$$
Moreover, $\Theta\circ\phi_{{\eta}_j}(\eta_j')=(0,1,-1)\in \partial \mathcal{D}_{1,1,2}$ and $\Theta\circ\phi_{{\eta}_j}(\eta_j)=(0,1,-2)\in \mathcal{E}_{1,1,2}$.  In addition,  the holomorphic map $\Psi $ defined by
\[
(z_1,z_2,w)\mapsto \Big( \frac{2}{1-w}z_1,\sqrt{\frac{2}{1-w}} ~z_2,  \frac{w+1}{1-w}\Big),
\]
is a biholomorphism from $\mathcal{E}_{1,1,2}$ onto the ellipsoid
\[
\mathcal{D}_{1,1,2}=\left \{(z_1,z_2, w)\in \mathbb C^3\colon |w|^2+|z_1|^2+ |z_2|^4<1\right\}.
\]

Finally, let us consider the sequence of biholomorphic map $f_j:=\Psi\circ \Theta\circ  \Phi_{\eta_j} \colon \Omega \to  f_j(\Omega)$. Since $\Psi\circ\Theta\circ\phi_{{\eta}_j}(\eta_j')=\Psi(0,1,-1)=(0,1,0)\in \partial \mathcal{D}_{1,1,2}$ and $\Psi\circ\Theta\circ\phi_{{\eta}_j}(\eta_j)=\Psi(0,1,-2)=\Big(0,\sqrt{\dfrac{2}{3}}, -\dfrac{1}{2}\Big)\in \mathcal{D}_{1,1,2}$. Furthermore, we have $\Psi(z_1,z_2,w)\to  (0,0,-1)$ as $ \mathcal{E}_{1,1,2}\ni (z_1,z_2,w)\to \infty$. As in the proof of Theorem \ref{maintheorem2}, we conclude that $ f_j(\Omega\cap U_0) $ and $ f_j(\partial \Omega\cap U_0) $ converge to $\mathcal{D}_{1,1,2} $ and $\partial \mathcal{D}_{1,1,2}$ in the Hausdorff sense, respectively. Therefore, this implies that
\[
\sigma_{\Omega\cap U_0}(\eta_j)=\sigma_{f_j(\Omega\cap U_0)}\Big(0,\sqrt{2/3}, -1/2\Big)>\dfrac{1}{2} \dfrac{\mathrm{dist}\Big(\big(0,\sqrt{2/3}, -1/2\big),\partial\mathcal{D}_{1,1,2}\Big)}{\mathrm{diam}(\mathcal{D}_{1,1,2})}>0
\]
for any $j$ big enough, where $\mathrm{diam}(\mathcal{D}_{1,1,2})$ denotes the diameter of the domain $\mathcal{D}_{1,1,2}$. Hence, the proof is now complete.
\end{proof}
\begin{remark} Consider the domain $E_{1,2,4}$ and the sequence $\{\eta_j\} \subset E_{1,2,4}$ as in Proposition \ref{pro2}. Then we have $\sigma_{E_{1,2,4}}(\eta_j)\not \to 1$ as $j\to\infty$, contrary to Theorem \ref{maintheorem2}.  Indeed, suppose, for the sake of contradiction, that $\sigma_{E_{1,2,4}}(\eta_j) \to 1$ as $j\to\infty$, then by the argument as in the proof of \cite[Theorem $2.1$]{NN20} the unit ball $\mathbb B^{n+1}$ is biholomorphically equivalent to $\mathcal{D}_{1,1,2}=\left \{(z_1,z_2, w)\in \mathbb C^3\colon |w|^2+|z_1|^2+ |z_2|^4<1\right\}$. Therefore, we arrive at a contradiction, as  $\mathcal{D}_{1,1,2}$ is not homogeneous.
\end{remark}

\section{The boundary behavior of the squeezing function near a weakly pseudoconvex point in $\mathbb C^2$}\label{S5}
\subsection{The spherically tangential convergence}\label{sub-sphere}
Let $\Omega$ be a domain in $\mathbb C^2$ and $\xi_0\in \partial \Omega$. Assume that $\partial\Omega$ is $\mathcal{C}^\infty$-smooth and pseudoconvex of D'Angelo finite type near $\xi_0$. After a change of variables, there are the coordinate functions $(z,w)$ such that $\xi_0=(0,0)$ and  $\rho(z,w)$, the local defining function for $\Omega$ near $\xi_0$, can be expanded near $(0,0)$ as follows:
\begin{equation}\label{def-0}
\rho(z, w) =\mathrm{Re}(w) + H(z) +v\varphi(v, z)+ O(|z|^{2m+1}|), 
\end{equation}
where $H$ is a real homogeneous subharmonic polynomial of degree $2m$, where $2m$ is the D'Angelo type of $\partial \Omega$ at $\xi_0$, not identically zero and without harmonic terms and $\varphi$ is a $\mathcal{C}^\infty$-smooth function defined in a neighborhood of the origin in $\mathbb R\times \mathbb C$ with $\varphi(0,0)=0$. Since the type is invariant under local biholomorphism and coincides with the maximal contact order at $(0,0)$ of germs of holomorphic curves with $\partial \Omega$. The pseudoconvexity of $\partial \Omega$ is equivalent to the subharmonicity of $H$ and the type $2m$ of $\partial \Omega$ at $\xi_0$ is then necessarily even.

 We recall the following definition.
\begin{define}[\cite{NNN25}]\label{spherically-convergence}
We say that a sequence $\{\eta_j=(\alpha_j,\beta_j)\}\subset  \Omega$ \emph{converges spherically $\frac{1}{2m}$-tangentially to $\xi_0$} if
\begin{itemize}
\item[(a)] $|\mathrm{Im}(\beta_j)|\lesssim |\mathrm{dist}(\eta_j,\partial \Omega)|$;
\item[(b)] $|\mathrm{dist}(\eta_j,\partial \Omega)|=o(|\alpha_{j}|^{2m})$;
\item[(c)]  $\Delta H(\alpha_{j})\gtrsim |\alpha_{j}|^{2m-2}$.
\end{itemize}
\end{define}
\begin{remark} In the case when $\Omega$ is a smooth pseudoconvex domain in $\mathbb C^2$, the condition $\mathrm{(c)}$ simply says that $\Omega$ is strongly pseudoconvex at $\eta_j'$ for every $j\in\mathbb N_{\geq 1}$, where $\{\epsilon_j\}\subset \mathbb R^+$ is a sequence such that $\eta_j':=(\alpha_j,\beta_j+\epsilon_j)\in \partial \Omega$ for all $j\in\mathbb N_{\geq 1}$. If $\Omega$ is strongly $h$-extendible at $x_0$, i.e.  $\Delta H(z)\gtrsim |z|^{2m-2}$, any sequence $\{\eta_j\}\subset \Omega$ converges spherically $\frac{1}{2m}$-tangentially to $\xi_0$ provided conditions (a) and (b) are satisfied.
\end{remark}

\subsection{Proof of Theorem \ref{maintheorem3}} Let $\Omega$ and $\xi_0\in \partial \Omega $ be as in the statement of Theorem \ref{maintheorem3}. As in Subsection \ref{sub-sphere}, one can find the coordinate functions $(z,w)$ such that $\xi_0=(0,0)$ and  $\rho(z,w)$ can be described near $(0,0)$ as follows:
\begin{equation}\label{def-0}
\rho(z, w) =\mathrm{Re}(w) + H(z) +v\varphi(v, z)+ O(|z|^{2m+1}|), 
\end{equation}
where $H$ is a real homogeneous subharmonic polynomial of degree $2m$ without harmonic terms and $\varphi$ is a $\mathcal{C}^\infty$-smooth function defined in a neighborhood of the origin in $\mathbb R\times \mathbb C$ with $\varphi(0,0)=0$. 

By hypothesis of Theorem \ref{maintheorem3}, let $\{\eta_j\}\subset\Omega$ be a sequence converging spherically $\frac{1}{2m}$-tangentially to $\xi_0$, and let us write $\eta_j=(\alpha_j,\beta_j)=(\alpha_{j},a_j+ib_j), \; \forall j\in \mathbb N_{\geq 1}$. In addition, without loss of generality we may assume that $\{\eta_j=(\alpha_j,\beta_j)\}\subset U_0^-:=U_0\cap\{\rho<0\}$ and one associates with a sequence of points $\eta_j'=(\alpha_{j}, a_j +\epsilon_j+i b_j)\in\partial \Omega$, for some sequence $\{\epsilon_j\}\subset \mathbb R^+$. Thus we have
\begin{itemize}
\item[(a)] $|b_j|\lesssim \epsilon_j$;
\item[(b)] $\epsilon_j=o(|\alpha_{j}|^{2m})$;
\item[(c)] $\Delta H(\alpha_{j})\gtrsim |\alpha_{j}|^{2m-2}$.
\end{itemize}

It follows from \cite[Section $3$]{Be03} (see also~\cite[Proposition $1.1$]{Cat89}) that,  for each point $\eta_j'$, there exists a biholomorphism $\Phi_{\eta_j'}$ of $\mathbb C^2$, $(z,w)=\Phi^{-1}_{\eta_j'}(\tilde z, \tilde w)$ defined by
$$
\Phi_{\eta_j'}^{-1}(z, w) =\Big(\alpha_j + z, a_j+\epsilon_j+i b_j + d_0(\eta_j')w + \sum_{1 \leq k \leq 2m} d_k(\eta_j')z^k\Big),
$$
where $d_0,\ldots,d_{2m}$ are $\mathcal{C}^\infty$-smooth functions defined in a neighborhood of the origin in $\mathbb C^{2}$ with $d_0(0,0)=1,d_1(0,0)=\cdots=d_{2m}(0,0)=0$, such that
\begin{equation}\label{Eq50}
\rho \circ \Phi_{\eta_j'}^{-1}(z, w)  = \mathrm{Re}(w) + \sum_{\substack{j+k \leq 2m \\ j,k > 0}} a_{j,k}(\eta_j')z^j\bar z^k + O(|z|^{2m+1} + |z||w|). 
\end{equation}

We first define
\begin{equation*}\label{Eq51}
\begin{split} 
A_l(\eta_j')&=\max  \left\{|a_{j,k}(\eta_j')|, \ j+k=l\right\}\ (  2\leq l \leq 2m).
\end{split}
\end{equation*}
Then,  for each $\delta>0$, one defines $\tau(\eta_j',\epsilon_j)$ as follows: 
\begin{equation*}\label{Eq6} 
\tau_j=\tau(\eta_j',\epsilon_j)=\min \left\{\big( \epsilon_j/A_l(\eta_j') \big)^{1/l},\ 2\leq l \leq 2m\right \}.
\end{equation*}
Since the type of $\partial \Omega$ at $\xi_0$ equals $2m$, $A_{2m}(\xi_0)\ne 0$. Thus, if $U_0$ is sufficiently small, then $|A_{2m}(\eta_j')|\geq c>0$ for all $\eta_j'\in U_0$. This gives  the inequality
\begin{equation*}\label{Eq7} 
\delta^{1/2}\lesssim \tau(\eta_j',\delta)  \lesssim \delta^{1/m}\  (\eta_j'\in U).
\end{equation*}

To finish the scaling procedure, let us define an anisotropic dilation $\Delta_j$ by 
\[
\Delta_j (z,w)=\left( \frac{z}{\tau_{j}},\frac{w}{\epsilon_j}\right),\; j\in \mathbb N_{\geq 1}.
\]
As in the proof of Theorem \ref{maintheorem2}, one sees that $\Delta_j\circ \Phi_{\eta_j'}(\eta_j')=(0,0)$ and $\Delta_j\circ \Phi_{\eta_j'}(\eta_j)=(0,-1/d_0(\eta_j'))\to(0,-1)$ as $j\to\infty$, since $d_0(\eta_j'))\to 1$ as $j\to\infty$. Furthermore, for each $j\in \mathbb N_{\geq 1}$, if we set $\rho_j(z,w)=\epsilon_j^{-1}\rho\circ \Phi_{\eta_j'}^{-1}\circ(\Delta_j)^{-1}(z,w)$, then \eqref{Eq50} implies that
\begin{equation*}\label{def-j}
\rho_j(z,w)=\mathrm{Re}(w)+ P_{\eta_j'}(z)+O(\tau(\eta_j',\epsilon_j)),
\end{equation*}
where
\begin{equation*}
\begin{split}
&P_{\eta_j'}(z):=\sum_{\substack{k,l\leq 2m\\
 k,l>0}} a_{k,l}(\eta_j') \epsilon_j^{-1} \tau_j^{k+l}z^k \bar z^l.
\end{split}
\end{equation*}

Next, if we write $\displaystyle H(z)=\sum_{j=1}^{2m-1} a_j z^j\bar z^{2m-j}$ and $z=|z| e^{i\theta}$, then we obtain $H(z)=|z|^{2m} g(\theta)$ for some function $g(\theta)$. Hence, as in \cite{BF78} one has
 $$
 \Delta H(z)=|z|^{2m-2} \left((2m)^{2} g(\theta)+g_{\theta\theta}(\theta)\right)\geq 0.
 $$
Moreover, \cite[Lemma $4.1$]{NNN25} implies that
$$
\frac{\partial^2 H(\alpha_j)}{\partial z\partial \bar z}\epsilon_j^{-1}\tau_j^{2}=(2m)^2g(\theta_j)+g_{\theta\theta}(\theta_j),\; \forall j\geq 1,
$$
where $\alpha_{j}=|\alpha_{j}|e^{\theta_j},\;j\geq 1$. Thanks to the condition (c), without loss of generality we may assume that  the limit $\displaystyle a:=  \lim_{j\to \infty} \dfrac{1}{2}\frac{\partial^2 H}{\partial z\partial \bar z}(\alpha_{j})\epsilon_j^{-1} \tau_{j}^2 $ exists.

A simple calculation shows that
$$
 a_{l,k-l}(\eta_j')=\frac{1}{k!}\frac{\partial^{k} \rho}{\partial z^l\partial \bar z^{k-l}}(\eta_j')=\frac{1}{k!}\frac{\partial^{k} H}{\partial z^l\partial \bar z^{k-l}}(\alpha_j)+ \frac{b_j}{k!}\frac{\partial^{k} \varphi}{\partial z^l\partial \bar z^{k-l}}(b_j,\alpha_j)+\cdots,\; 
$$
for all $j\in \mathbb N_{\geq 1}$, $2\leq k\leq 2m$, and $0\leq l\leq k$, where the dots denote remainder terms. Since $H$ is a homogeneous subharmonic polynomial of degree $2m$, it follows that $\displaystyle \Big|\frac{\partial^{k} H}{\partial z^l\partial \bar z^{k-l}}(\alpha_j)\Big|\lesssim |\alpha_{j}|^{2m-k}$ for $2\leq k\leq 2m$. In addition, since $|b_j|\lesssim \epsilon_j=o( |\alpha_{j}|^{2m})$ one has  $|a_{l,k-l}(\eta_j')|\lesssim |\alpha_{j}|^{2m-k}$ for $2\leq k\leq 2m$. This implies that $A_k(\eta_j')\lesssim |\alpha_{j}|^{2m-k}$, and hence one gets
$$
\big(\epsilon_j/ A_k(\eta_j') \big)^{1/k}\gtrsim\big(\epsilon_j/ |\alpha_{j}|^{2m-k}\big)^{1/k}=|\alpha_{j}| \big(\epsilon_j/ |\alpha_{j}|^{2m}\big)^{1/k},\;  2\leq k\leq 2m.
$$
Moreover, since $\epsilon_j=o( |\alpha_{j}|^{2m})$ and $|\alpha_{j}| \big(\epsilon_j/ |\alpha_{j}|^{2m}\big)^{1/2}=o\left(|\alpha_{j}| \big(\epsilon_j/ |\alpha_{j}|^{2m}\big)^{1/k}\right)$ for all $k\geq3$, it follows that  
$$
\tau_j=\big(\epsilon_j/ A_2(\eta_j') \big)^{1/2}\approx|\alpha_{j}| \big(\epsilon_j/ |\alpha_{j}|^{2m}\big)^{1/2}.
$$

Now, we establish the convergence of the sequence $\{\Delta_j\circ \Phi_{\eta_j'}(U_0^-)\}_{j=1}^{\infty}$. Indeed, a computation shows that
\begin{align*}
| a_{l,k-l}(\eta_j')| \epsilon_j^{-1}\tau_j^{k}\approx\left |\frac{\partial^k H}{\partial z^l\partial \bar z^{k-l}}(\alpha_j)\right| \epsilon_j^{-1}\tau_j^{k}&\lesssim |\alpha_j|^{2m-k} \epsilon_j^{-1}\tau_j^{k}=|\alpha_j|^{2m} \epsilon_j^{-1}\Big(\frac{\tau_j}{|\alpha_j|}\Big)^{k}\\
&\lesssim \frac{|\alpha_j|^{2m}}{ \epsilon_j}\Big(\frac{\epsilon_j}{|\alpha_j|^{2m}}\Big)^{k/2}=\Big(\frac{\epsilon_j}{|\alpha_j|^{2m}}\Big)^{k/2-1}.
\end{align*}
This yields that $ a_{l,k-l}(\eta_j')| \epsilon_j^{-1}\tau_j^{k}\to 0$ as $j\to\infty$ for $3\leq k\leq 2m$ and 
$$
  \lim_{j\to \infty}a_{1,1}(\eta_j')\epsilon_j^{-1} \tau_{j}^2 =\lim_{j\to \infty} \dfrac{1}{2}\frac{\partial^2 H}{\partial z\partial \bar z}(\alpha_{j})\epsilon_j^{-1} \tau_{j}^2=a >0.
$$
Altogether, we conclude that, after taking a subsequence if necessary,  the sequence $\{ \rho_j\}$ converges on compacta to the following function
$$
\hat\rho(z,w):=\mathrm{Re}(w)+a|z|^2,
$$
where $\displaystyle a=\frac{1}{2} \lim_{j\to \infty} \frac{\partial^2 H}{\partial z\partial \bar z}(\alpha_{j})\epsilon_j^{-1} \tau_{j}^2 > 0$. Therefore, passing to a subsequence if necessary, we may assume that the sequences $\Omega_j:=\Delta_j\circ \Phi_{\eta_j'}(\Omega) $ and $\Delta_j\circ \Phi_{\eta_j'}(\Omega\cap U_0)$ converge normally to the Siegel half-space
$$
M_{a}:=\left \{( z,w)\in \mathbb C^2\colon \hat\rho(z,w)=\mathrm{Re}(w)+a|z|^2<0\right\}.
$$

The remainder of the proof is to estimate $\sigma_{\Omega\cap U_0}(\eta_j)$. To do this, let us first define the linear transformation $\Theta$, given by 
\[ 
\tilde w= w;  \tilde z=\sqrt{a}\, z,
\]
maps $M_{a}$ onto the Siegel half-space
$$
\mathcal{U}_2:=\{(z,w)\in \mathbb{C}^2 \colon\mathrm{Re}(w) +|z|^2<0\}.
$$
In addition,  the holomorphic map $\Psi $, defined by
\[
(z,w)\mapsto  \Big( \frac{2z}{1-w},  \frac{w+1}{1-w}\Big),
\]
is a biholomorphism from $\mathcal{U}_2$ onto $\mathbb B^2$.

Next, let us consider the sequence of biholomorphic map $f_j:=\Psi\circ \Theta\circ \Delta_j\circ \Phi_{\eta_j'} \colon \Omega \to  f_j(\Omega)=\Psi\circ \Theta(\Omega_j)$. Notice that $f_j(\Omega \cap U_0)$ and $f_j(\partial \Omega \cap U_0)$ converge normally to $\mathbb{B}^{2}$ and $\partial \mathbb{B}^{2}$, respectively. Moreover, since $\Theta(0,-1) =(0,-1)$ and $\Psi(0,-1)=(0,0)$, it follows that
\[
f_j(\eta_j) = \Psi\circ \Theta(0,-1/d_0(\eta_j')) = \Psi(0,-1/d_0(\eta_j')) = \left(0, \frac{1-1/d_0(\eta_j')}{1+1/d_0(\eta_j')}\right)  \to (0, 0) \quad \text{as } j \to \infty.
\]
Therefore, by a similar argument as in the proof of Theorem \ref{maintheorem2}, we conclude that for a sufficiently small $\epsilon > 0$, there exists $j_0 \in \mathbb{N}_{\geq 1}$ such that
\[
B\big((0,0), 1 - \epsilon\big) \subset F_j(U_0\cap \Omega) \subset B\big((0,0), 1 + \epsilon\big), \, \forall j \geq j_0,
\]
where $F_j(\cdot) := f_j(\cdot) - f_j(\eta_j)$ for all $j \geq j_0$. Since $ F_j(\eta_j) = 0,\; \forall j \geq j_0$, it follows that
\[
\sigma_{\Omega\cap U_0}(\eta_j) \geq \frac{1 - \epsilon}{1 + \epsilon}, \quad \forall j \geq j_0.
\]
Since $\epsilon>0$ is arbitrary, we conclude that $\displaystyle \lim_{j\to \infty}\sigma_{\Omega\cap U_0}(\eta_j)=1$, and  the proof of Theorem~\ref{maintheorem3} is now complete. \hfill $\Box$

Let us write $\eta_j=(\alpha_j,\beta_j)$ for $j\in\mathbb{N}_{\geq 1}$. Without condition (a) in Definition \ref{spherically-convergence}, for some domain $\Omega$ and some sequence $\{\eta_j\}\subset \Omega$ where $\mathrm{Im}(\beta_j)$ has a significant contribution to $\frac{\partial^2 \rho}{\partial z\partial \bar{z}}(\eta_j)$, the quantity $\frac{\partial^2 \rho}{\partial z\partial \bar{z}}(\eta_j)$ may differ significantly from $\frac{\partial^2 P}{\partial z\partial \bar{z}}(\eta_j)$. However, the limiting model may still be biholomorphically equivalent to the unit ball. The following example demonstrates this phenomenon.
\begin{example} \label{Ex5.1}
Let $\mathcal{G}$ be the domain in $\mathbb{C}^2$ defined by
$$
\mathcal{G}:=\left\{(z,w)\in \mathbb{C}^2\colon \rho(z,w):=\mathrm{Re}(w)+ |z|^4+|\mathrm{Im}(w)|^2 |z|^2 <0\right\}.
$$
Consider the sequences $\eta_j=\Big(\frac{1}{j^{1/4}},-\frac{2}{j}-\frac{1}{j^2}+\frac{i}{j^{1/4}}\Big)\in \mathcal{G}$ and $\eta_j'=\Big(\frac{1}{j^{1/4}},-\frac{2}{j}+\frac{i}{j^{1/4}}\Big)\in \partial \mathcal{G}$ for all $j\in \mathbb{N}_{\geq 1}$. We have $\rho(\eta_j)=-\frac{1}{j^2}<0$ for all $j\in \mathbb{N}_{\geq 1}$, so we set $\epsilon_j=\frac{1}{j^2}\approx \mathrm{dist}(\eta_j,\partial \mathcal{G})$ and $b_j=\frac{1}{j^{1/4}}$ for all $j\in \mathbb{N}_{\geq 1}$. 

We first observe that $|b_j|=\frac{1}{j^{1/4}}\not \lesssim \epsilon_j$. Therefore, condition $(a)$ in Definition \ref{spherically-convergence} is not satisfied, and thus $\{\eta_j\}$ does not converge spherically $\frac{1}{4}$-tangentially to $\xi_0=(0,0)$. Although condition $(b)$ still holds since $\epsilon_j=o\big(\frac{1}{j}\big)=o(|\alpha_j|^4)$ with $\alpha_j=\frac{1}{j^{1/4}}$, the scaling method in the proof of Theorem \ref{maintheorem3} cannot be employed.

We now introduce an alternative scaling method. Indeed,  the Taylor expansion of the function $\rho$ in a neighbourhood of $\eta_j'$ has the form
\begin{align*}
\rho(z,w)&=\mathrm{Re}(w)+ \left|z-\frac{1}{j^{1/4}}+\frac{1}{j^{1/4}}\right|^4+\left|v-\frac{1}{j^{1/4}}+\frac{1}{j^{1/4}}\right|^2 \left|z-\frac{1}{j^{1/4}}+\frac{1}{j^{1/4}}\right|^2 \\
              &=\mathrm{Re}\Big(w+\frac{2}{j}\Big)+\frac{6}{j^{3/4}} \mathrm{Re}\Big(z-\frac{1}{j^{1/4}}\Big)+\frac{5}{j^{1/2}} \left|z-\frac{1}{j^{1/4}}\right|^2+\frac{2}{j^{1/2}} \mathrm{Re}\Big(\Big(z-\frac{1}{j^{1/4}}\Big)^2\Big)\\
& \quad +O\Big(\frac{1}{j^{3/4}}\Big(v-\frac{1}{j^{1/4}}\Big)+\frac{1}{j^{1/4}} \left|z-\frac{1}{j^{1/4}}\right|^3\Big),
\end{align*}
where $v=\mathrm{Im}(w)$.

To apply the scaling method, we define the scaling parameter $\tau_j:=\frac{1}{j^{3/4}}$ for all $j\in \mathbb{N}_{\geq 1}$. We then define a sequence of polynomial automorphisms $\phi_{{\eta}_j}^{-1}$ of $\mathbb{C}^2$ given by
\begin{align*}
 z&=\frac{1}{j^{1/4}}+\tau_j \tilde{z};\\
w&=\epsilon_j \tilde{w}-\frac{2}{j}+\frac{i}{j^{1/4}}-\frac{6}{j^{3/4}} \tau_j \tilde{z}-\frac{2}{j^{1/2}}  \tau_j^2 \tilde{z}^2.
\end{align*}
Since $\tau_j=\frac{1}{j^{3/4}}=o\big(\frac{1}{j^{1/4}}\big)$, we obtain
\begin{equation*}
\epsilon_j^{-1}\rho\circ \phi_{{\eta}_j} ^{-1}(\tilde{z},\tilde{w})= \mathrm{Re}(\tilde{w}) + 5|\tilde{z}|^2 +O\Big(\frac{1}{j^{1/2}}\Big).
\end{equation*} 
This ensures that $\Omega_j:=\phi_{{\eta}_j}(\mathcal{G}\cap U_0)$, where $U_0$ is a sufficiently small neighborhood of $(0,0,0)$, converges normally to the model $\mathcal{F}:=\{(\tilde{z},\tilde{w})\in \mathbb{C}^2\colon \mathrm{Re}(\tilde{w}) + 5|\tilde{z}|^2<0 \}$, which is biholomorphically equivalent to the unit ball $\mathbb{B}^2$, and $\phi_{{\eta}_j}(\eta_j)=(0,-1)\in \mathcal{F}$ for all $j\geq 1$. Therefore, by following the proof of Theorem~\ref{maintheorem3} we conclude that $\sigma_{\mathcal{G}\cap U_0}(\eta_j)\to 1$ as $j\to \infty$. \hfill $\Box$

\end{example}

Next, the following example illustrates spherically $\frac{1}{2m}$-tangential convergence.
\begin{example}\label{Kohn-Nirenberg}
Let $\Omega_{KN}$ be the Kohn-Nirenberg domain in $\mathbb{C}^2$ that does not admit a holomorphic support function (see \cite{KN73}), defined by
$$
\Omega_{KN}:=\left\{(z,w)\in \mathbb{C}^2\colon \mathrm{Re}(w)+ |z|^8+\frac{15}{7}|z|^2\mathrm{Re}(z^6)<0\right\}.
$$
Let us consider a bounded domain $\Omega$ with $(0,0)\in \partial \Omega$ such that $\Omega\cap U_0=\Omega_{KN}\cap U_0$ for some neighbourhood $U_0$ of $(0,0)$ in $\mathbb{C}^2$. We denote by $\rho(z,w)=\mathrm{Re}(w)+ |z|^8+\frac{15}{7}|z|^2\mathrm{Re}(z^6)$ and $P(z)=|z|^8+\frac{15}{7}|z|^2\mathrm{Re}(z^6)$. It is easy to see that $\Delta P(z)=4(16|z|^6+15\mathrm{Re}(z^6))\geq 4|z|^6$, and hence $\partial \Omega$ is strongly $h$-extendible at $(0,0)$.

We first consider a sequence $\eta_j=\Big(\frac{1}{j^{1/8}},-\frac{22}{7j}-\frac{1}{j^2}\Big)\in \Omega$ for every $j\in \mathbb{N}_{\geq 1}$. Then the sequence $\left\{\left(\frac{1}{j^{1/8}},-\frac{22}{7j}-\frac{1}{j^2}\right)\right\}$ converges spherically $\frac{1}{8}$-tangentially to $(0,0)$. Moreover, we have $\rho(\eta_j)=-\frac{22}{7j}-\frac{1}{j^2}+\frac{22}{7j}=-\frac{1}{j^2}\approx -\mathrm{dist}(\eta_j,\partial \Omega_{KN})$. Setting $\epsilon_j=|\rho(\eta_j)|=\frac{1}{j^2}$, a computation shows that
\begin{align*}
&\rho(z,w)\\
&=\mathrm{Re}(w)+ \Big|\Big(z-\frac{1}{j^{1/8}}\Big)+\frac{1}{j^{1/8}}\Big|^8+\frac{15}{7}\Big|\Big(z-\frac{1}{j^{1/8}}\Big)+\frac{1}{j^{1/8}}\Big|^2\mathrm{Re}\Big(\Big(\Big(z-\frac{1}{j^{1/8}}\Big)+\frac{1}{j^{1/8}}\Big)^6\Big)\\
&=\mathrm{Re}(w)+\frac{1}{j}+\frac{8}{j^{7/8}} \mathrm{Re}\Big(z-\frac{1}{j^{1/8}}\Big)+\frac{16}{j^{3/4}} \Big|z-\frac{1}{j^{1/8}}\Big|^2+\frac{12}{j^{3/4}} \mathrm{Re}\Big(\Big(z-\frac{1}{j^{1/8}}\Big)^2\Big)\\
&\quad+\frac{15}{7}\left[\frac{1}{j}+\frac{8}{j^{7/8}}\mathrm{Re}\Big(z-\frac{1}{j^{1/8}}\Big)+\frac{21}{j^{3/4}} \mathrm{Re}\Big(\Big(z-\frac{1}{j^{1/8}}\Big)^2\Big)+\frac{7}{j^{3/4}}\Big|z-\frac{1}{j^{1/8}}\Big|^2\right]+\cdots\\
&=\mathrm{Re}(w)+\frac{22}{7j}+\frac{176}{7j^{7/8}} \mathrm{Re}\Big(z-\frac{1}{j^{1/8}}\Big)+\frac{57}{j^{3/4}} \mathrm{Re}\Big(\Big(z-\frac{1}{j^{1/8}}\Big)^2\Big)+\frac{31}{j^{3/4}}\Big|z-\frac{1}{j^{1/8}}\Big|^2\\
&\quad +O\Big(\frac{1}{j^{5/8}}\Big|z-\frac{1}{j^{1/8}}\Big|^3\Big).
\end{align*}

To define an anisotropic dilation, let us denote $\tau_j:=\tau(\eta_j)=\frac{1}{j^{5/8}}$ for all $j\in \mathbb{N}_{\geq 1}$. Now we introduce a sequence of polynomial automorphisms $\phi_{{\eta}_j}^{-1}$ of $\mathbb{C}^2$, given by
\begin{align*}
 z&=\frac{1}{j^{1/8}}+\tau_j \tilde{z};\\
w&=\epsilon_j \tilde{w}-\frac{22}{7j}-\frac{176}{7j^{7/8}} \tau_j \tilde{z}-\frac{57}{j^{3/4}}  \tau_j^2 \tilde{z}^2.
\end{align*}
Therefore, since $\tau_j=\frac{1}{j^{5/8}}=o\big(\frac{1}{j^{1/8}}\big)$, we have
\begin{equation*}
\epsilon_j^{-1}\rho\circ \phi_{{\eta}_j} ^{-1}(\tilde{z},\tilde{w})= \mathrm{Re}(\tilde{w}) + 31|\tilde{z}|^2 +O\Big(\frac{1}{j^{1/2}}\Big).
\end{equation*} 
This implies that $\Omega_j:=\phi_{{\eta}_j}(\Omega\cap U)$, where $U$ is a sufficiently small neighborhood of $(0,0)$, converges normally to the model $\mathcal{H}:=\{(\tilde{z},\tilde{w})\in \mathbb{C}^2\colon \mathrm{Re}(\tilde{w}) + 31|\tilde{z}|^2<0 \}$, which is biholomorphically equivalent to $\mathbb{B}^2$, and $\phi_{{\eta}_j}(\eta_j)=(0,-1)\in \mathcal{H}$ for all $j\geq 1$. By arguments as in the proof of Theorem \ref{maintheorem3}, we conclude that $\sigma_{\Omega\cap U}(\eta_j)\to 1$ as $j\to \infty$.
\hfill $\Box$ 
\end{example}

To complete this section, we introduce the following example, which demonstrates the case when $\Delta P(\alpha_j)=0,\;\forall j\in \mathbb N_{\geq 1}$.

\begin{example}\label{rmk5.2} 
As in \cite{NNN25}, instead of $\Omega_{KN}$ we consider a bounded domain $\Omega$ such that $\Omega\cap U_0=\widetilde{\Omega}_{KN}\cap U_0$, where $U_0$ is a neighbourhood of the origin in $\mathbb{C}^2$ and 
 $$
\widetilde{\Omega}_{KN}:=\left\{(z,w)\in \mathbb{C}^2\colon \mathrm{Re}(w)+ |z|^8-\frac{16}{7}|z|^2\mathrm{Re}(z^6)<0\right\}.
$$
Let $P(z)=|z|^8-\frac{16}{7}|z|^2\mathrm{Re}(z^6)$ and $\alpha_j=1/j^{1/8}$ for all $j\geq 1$. Then $\Delta P(\alpha_j)=0$ for all $j\geq 1$. We now consider the sequence $\left\{\left(\frac{1}{j^{1/8}},\frac{9}{7j}-\frac{1}{j^2}\right)\right\}\subset \Omega$ that converges $\frac{1}{8}$-tangentially but not spherically $\frac{1}{8}$-tangentially to $(0,0)$. Then let us define $\epsilon_j=\frac{1}{j^2},\tau_j=\frac{1}{j^{3/8}}$ for all $j\in \mathbb N_{\geq 1}$. Therefore, by arguments as in Example \ref{Kohn-Nirenberg}, we conclude that our model is given by 
 $$
 \mathcal{A}:=\left\{(\tilde{z},\tilde{w})\in \mathbb{C}^2\colon \mathrm{Re}(\tilde{w})+36 |\tilde{z}|^4-48|\tilde{z}|^2\mathrm{Re}(\tilde{z}^2)<0\right\}.
 $$
(For more details, see Example $5.1$ in \cite{NNN25}.) However, it is not clear that $D$ is biholomorphically equivalent to the domain 
$$
 \mathcal{B}:=\left\{(\tilde{z},\tilde{w})\in \mathbb{C}^2\colon |\tilde{w}|^2+36 |\tilde{z}|^4-48|\tilde{z}|^2\mathrm{Re}(\tilde{z}^2)<1\right\}.
$$
Therefore, the scaling method as in the proof of Theorem \ref{maintheorem3} may not be applicable. Furthermore, it remains to be seen whether $ \mathcal{A}$ is biholomorphically equivalent to a bounded domain (note that even the domain $ \mathcal{B}$ is unbounded), and so $\sigma_{ \mathcal{A}}$ may not be defined.

\end{example}
\section*{Acknowledgments}
The author was supported by the Vietnam National Foundation for Science and Technology Development (NAFOSTED) under grant number 101.02-2021.42. 
We would like to thank Professor Shichao Yang for providing a counterexample that helped us complete the proof of Theorem \ref{maintheorem1}.
%We are grateful to an anonymous referee for his/her very careful reading and constructive comments that significantly improve our exposition.
%\end{acknowledgement}

\bibliographystyle{plain}

\end{document}